\documentclass[12pt,a4]{amsart}
\usepackage{a4wide}

\usepackage{amsmath,amscd}
\usepackage{amsmath}
\usepackage{amsfonts}
\usepackage{amsthm}
\usepackage{xcolor}
\usepackage{hyperref}

\usepackage{enumitem}

\theoremstyle{plain}
\newtheorem{Thm}{Theorem}[section]
\newtheorem{Cor}[Thm]{Corollary}
\newtheorem{Lem}[Thm]{Lemma}
\newtheorem{Prop}[Thm]{Proposition}

\newtheorem{Claim}{Claim}
\theoremstyle{definition}		
\newtheorem{Def}[Thm]{Definition}
\newtheorem{Rem}[Thm]{Remark}
\newtheorem{Ex}[Thm]{Example}
\newtheorem{qn}[Thm]{Question}

\newcommand{\ar}{\mbox{${\mathcal{R}}$}}
\newcommand{\el}{\mbox{${\mathcal{L}}$}}
 
\newcommand{\eh}{\mbox{${\mathcal{H}}$}}
\newcommand{\jay}{\mbox{${\mathcal{J}}$}}
\newcommand{\dee}{\mbox{${\mathcal{D}}$}}

\newcommand{\art}{\mbox{$\widetilde{\mathcal{R}}$}}
\newcommand{\elt}{\mbox{$\widetilde{\mathcal{L}}$}}
\newcommand{\eht}{\mbox{$\widetilde{\mathcal{H}}$}}

\newcommand{\arte}{\mbox{$\widetilde{\mathcal{R}}_E$}}
\newcommand{\elte}{\mbox{$\widetilde{\mathcal{L}}_E$}}
\newcommand{\ehte}{\mbox{$\widetilde{\mathcal{H}}_E$}}

\renewcommand{\phi}{\varphi}

\renewcommand{\a}{\mathbf{a}}

\renewcommand{\b}{\mathbf{b}}

\newcommand{\N}{\mathbb{N}}
\newcommand{\Z}{\mathbb{Z}}

\newcommand{\BR}{\operatorname{BR}}
\renewcommand{\max}{\operatorname{max}}

\renewcommand{\Im}{\operatorname{Im}}

\DeclareMathOperator{\ann}{\mathbf{r}}
\DeclareMathOperator{\supp}{supp}
\newcommand{\tuple}[1]{\mbox{{\boldmath{$#1$}}}}

\newenvironment{txitemize}{\begin{itemize}[leftmargin=7mm,itemsep=1mm]}{\end{itemize}}
\newenvironment{thmenumrom}{\begin{enumerate}[leftmargin=10mm,itemsep=1mm,label=\textup{(\roman*)}]}{\end{enumerate}}
\newenvironment{thmenumarab}{\begin{enumerate}[leftmargin=10mm,itemsep=1mm,label=\textup{(\arabic*)}]}{\end{enumerate}}
\newenvironment{thmenumalph}{\begin{enumerate}[leftmargin=10mm,itemsep=1mm,label=\textup{(\alph*)}]}{\end{enumerate}}

\begin{document}

\title[Coherency and constructions]{Coherency   and  constructions for monoids}

\thanks{The authors acknowledge the support of the Engineering and Physical Sciences Research Council (EPSRC) grant  EP/I032312/1.
Research also partially supported by the Hungarian Scientific Research Fund (OTKA) grant  PD115705. The research of Yang Dandan
was supported by grant  20170604 of the Young Talents Project of Shaanxi
Association for Science and Technology,  by grants  20103176174 and
 JB180714 of the Fundamental Research Funds for the Central Universities, and by grant 2020JM-178 of Shaanxi Province Basic Research Program of Natural Science.
}

\date{\today}

\author[Y. Dandan]{Yang Dandan}
\address{School of Mathematics and Statistics, Xidian University, Xi'an 710071, P. R. China }
\email{ddyang@xidian.edu.cn}

\author[V. Gould]{Victoria Gould}
\address{Department of Mathematics, University of York, Heslington, York, YO10 5DD, UK}
\email{victoria.gould@york.ac.uk}

\author[M. Hartmann]{Mikl\'os Hartmann}
\email{hmikklos@gmail.com}

\author[N. Ru\v{s}kuc]{Nik Ru\v{s}kuc} 
\address{School of Mathematics and Statistics, University of St Andrews, St Andrews, KY16 9SS, UK}
\email{nik.ruskuc@st-andrews.ac.uk}

\author[R-E Zenab]{Rida-E Zenab}
\address{Department of Mathematics
Sukkur IBA University, Pakistan}
\email{ridaezenab@iba-suk.edu.pk}


\begin{abstract}
A monoid $S$ is {\em right coherent} if every finitely generated subact of every finitely
presented right $S$-act is finitely presented.
This is a finiteness condition, and we investigate whether or not it is preserved under some standard algebraic and semigroup theoretic constructions: subsemigroups, homomorphic images, direct products,
Rees matrix semigroups, including Brandt semigroups, and Bruck--Reilly extensions.
We also investigate the relationship with the property of being weakly right noetherian, which requires all right ideals of $S$ to be finitely generated.
\end{abstract}

\maketitle

\section{Introduction}\label{sec:intro}

A {\em finitary property} for a class of algebraic structures $\mathcal{A}$ is a property that is certainly satisfied by all finite algebras in $\mathcal{A}$. Of course, one hopes that the property will also be satisfied by some infinite members of $\mathcal{A}$, in such a way that it forces finite-like behaviour. 
Studying algebras via their finitary properties is a classical tool, introduced by  
Noether and Artin in the early part of the last century in the context of descending and ascending chain conditions for `classical' algebras such as rings. It  has  very broad implications today in areas ranging from structure theory to decidability problems.  

In this paper we work with the class of monoids, and the finitary property is that of right coherency. We arrive at this property  by considering representations of  monoids  by mappings of sets, as we now describe. Let $S$ be a monoid with identity $1$. A  {\em right $S$-act}
is a set $A$ together with a map $A\times S\rightarrow A$ 
where $(a,s)\mapsto as$, such that
for all $a\in A$ and $s,t\in S$ we have $a1=a$ and $a(st)=(as)t$. 
A right action of $S$ on $A$ may also be viewed as a morphism from  $S$ to the monoid of all
mappings of $A$ to itself (with left-to-right composition). 
We also have the dual notion of a {\em left $S$-act}. 
Right $S$-acts over a monoid $S$  may be regarded as  the non-additive analogue of right $R$-modules over a (unital) ring $R$.
Although the study of the two notions diverges considerably once technicalities set in, one can often begin by forming analogous notions and asking corresponding questions.

A monoid $S$ is said to be {\em right coherent} if every finitely generated subact of every finitely presented right $S$-act is finitely presented.
{\em Left coherency} is defined dually; $S$ is {\em coherent} if it is both right 
and left coherent. These notions are analogous to their name-sakes for a  ring $R$  \cite{chase:1960}
(where, of course, $S$-acts are replaced by $R$-modules).
Coherency is a finitary condition for both rings and monoids.
As demonstrated by Wheeler
\cite{wheeler:1976}, it is intimately related to the model theory of  $S$-acts  and  $R$-modules. Indeed, the first order theory of right $S$-acts (respectively, right $R$-modules) has a {\em model companion}  if and only if 
$S$ (respectively, $R$), is right coherent. Having a model companion ensures that the first order theory is  well behaved, in particular, it is amenable to the application of concepts of stability  \cite{eklof:1970,gould:1987,ivanov:92}. 

This paper is the third in a recent series  \cite{FMonoid,GM} investigating coherency of monoids. 
Its predecessors were concerned with free monoids in certain varieties of monoids and unary monoids. 
In particular, \cite{FMonoid} showed that all free monoids are coherent, building upon the earlier observation in \cite{gould:1992} that free commutative monoids are coherent, and resolving an open question from that paper.  
This theme is continued in \cite{GM}, where it was shown that any free left ample monoid is coherent, while free inverse monoids and free ample monoids of rank $>1$ are not.
In this paper we change tack somewhat, and note that little is known concerning the preservation of right coherency of monoids under standard algebraic constructions. Our primary aim is to start investigations exploring this strand. This also entails establishing  connections between right coherency and  finitary properties relating to the lattice of right ideals of a monoid.

In what follows we outline the organisation and content of the  paper.
In Section \ref{sec:prelim} we introduce the background material
needed for the rest of the paper.
This in particular includes the relationship between coherency and right congruences.
In the case of rings, coherency can be characterised by properties of certain right ideals.
Although monoids have right ideals, where a right ideal $I$ of $S$ is a subset such that
$IS\subseteq I$, they do not play the same role as in the case for rings, 
because we do not obtain \emph{all} single generator right $S$-acts by factoring out by a right ideal. 
In this regard
the natural notion  corresponding  to a right ideal in a ring is that of a 
{\em right congruence} of a monoid, that is, an equivalence relation $\rho$ on  $S$ 
such that $ac\,\rho\, bc$ for all $a,b,c\in S$ with $a\,\rho\, b$. 
Right ideals give rise to right congruences, but not all right congruences are obtained in this way.
In Section \ref{sec:prelim} we will review the standard conditions expressing coherency in terms of properties of congruences, as well as all other background material needed for the rest of the paper.

In Section \ref{sec:noeth} we explore further the link between congruences and coherency, and relate
the latter with another finiteness condition, that of being \emph{weakly right noetherian}.
A  monoid $S$ has this property if every
right ideal is finitely generated, and is {\em right noetherian} if every right congruence is finitely generated. 
From Normak \cite{normak:1977}, a right noetherian monoid is right coherent. 
In the parallel situation for rings, it is well known that the notions of being right noetherian and of being weakly right noetherian coincide, and that they imply right coherency (see, for example, \cite{rotman:1979}).  
On the other hand, for  a monoid, the notion of being weakly right noetherian is easily seen to be weaker than that of being right noetherian.  For example, any group is weakly right noetherian, 
but, since (finitely generated) right congruences correspond to (finitely generated) subgroups, 
it follows that not all groups are right noetherian.  Moreover,  \cite[Example 3.1]{gould:1992} tells us that, in general, a weakly right noetherian monoid need not be right coherent. 
However, if we restrict our attention to \emph{regular} monoids $S$, meaning that for all $a\in S$ there exists $b\in S$ such that $a=aba$, then we can demonstrate a stronger link  between right ideals and coherency (Theorem \ref{prop:wrnregular}), from which we are able to deduce:

\begin{txitemize}
\item
Any weakly right noetherian regular monoid is right coherent (Corollary \ref{cor:wrncoherent}).
\end{txitemize}

It is known that  there are right coherent monoids that are not weakly right noetherian, for example, any free monoid of rank greater than one \cite{FMonoid}. 
We exhibit two further examples at the end of Section \ref{sec:noeth}, both of them regular, in order to demonstrate the independence of various conditions we consider.

In Section~\ref{sec:down} 
we show that certain natural  submonoids of right coherent monoids are right coherent. In particular:

\begin{txitemize}
\item
For any right coherent monoid $S$, any monoid ideal and any monoid $\jay$-class are right coherent (Corollary \ref{cor:ideal} and Theorem \ref{thm:jclasses}).
\end{txitemize}

In the final part of Section \ref{sec:down} and the entire Section \ref{sec:up} we focus on Brandt monoids 
$\mathcal{B}(M,I)^1$, for which we can present a complete picture:

\begin{txitemize}
\item
A monoid $M$ is right coherent if and only if any Brandt monoid $\mathcal{B}(M,I)^1$ is right coherent
(Corollary \ref{cor:btdown} and Theorem \ref{thm:brandt}).
\end{txitemize}

We also consider two related constructions -- Rees matrix semigroups and Bruck--Reilly extensions -- but for them are only able to present partial results in Proposition \ref{prop:cs1} and Corollaries \ref{cor:BRG}, \ref{cor:br}.

 Section~\ref{sec:dp} is devoted to right coherency of direct products. 
 We exhibit an example which shows that direct products do not preserve coherency in general, and
 then prove:

\begin{txitemize}
\item
If $S$ is a right coherent monoid and $T$ is a finite monoid, then $S\times T$ is right coherent
(Theorem \ref{thm:SxTfin}).
\end{txitemize}

We regard the work in this paper as an opener in the discussion of the behaviour of coherency with regard to natural semigroup constructions, and its relation to other finitary properties. As we  proceed we suggest a number of  directions for   future work. As far as possible we have attempted to make the material self contained, but refer the reader to  \cite{howie:1995} and \cite{kkm:2000} 
 for more details of the terminology.

\section{Preliminaries}\label{sec:prelim}

Let $S$ be a monoid with identity $1$. We  allow $\emptyset$ to be a subact of any $S$-act.   The monoid $S$ may be regarded as a right $S$-act over itself, and as such its subacts are exactly its right ideals, so that, in particular, we allow $\emptyset$ to be a right ideal. In the category of right $S$-acts, a morphism  between right $S$-acts $A$ and $B$ is a map
$\theta:A\rightarrow B$ such that $(as)\theta=(a\theta)s$ for all $a\in A$ and $s\in S$; we refer to $\theta$ as an {\em $S$-morphism}. Since $S$-acts form a variety of algebras, the free $S$-act $F_S(X)$  over any set $X$ exists. The following describes its  very transparent structure. 

\begin{Prop}[{\cite[{Construction 1.5.14}]{kkm:2000}}] \label{prop:free}
Let $X$ be a set and $S$ a monoid.
Then $F_S(X)=X \times S$ with the action of $S$ defined as $(x,s) \cdot t=(x,st)$.
\end{Prop}

For simplicity in the above we abbreviate $(x,s)$ to $xs$ and $x1$ to $x$. 

For a right $S$-act $A$  and $a\in A$ define $aS:=\{ as:s\in S\}$.
Then $A$ is {\em finitely generated} if there exist $a_1,\dots,a_n\in A$ such that
\[A=a_1S\cup a_2S\cup\hdots \cup a_nS.\]
This is equivalent to $A$ being an  image of a finitely generated free $S$-act $F_S(X)$ where $|X|=n$, under an $S$-morphism. Thus $A$ is $S$-isomorphic to a  quotient of  $F_S(X)$, as we now explain. 

\begin{Def}\label{def:cong}
Let $A$ be a right $S$-act.
An equivalence relation $\rho$ on $A$ is a {\em right $S$-act congruence on $A$}  (\emph{congruence} for short, when there is no danger of confusion) if $a \ \rho\ b$ implies $as\ \rho\ bs$ for all $a,b \in A$ and $s\in S$.
\end{Def}

If $\rho$ is a congruence on a right $S$-act $A$, then the quotient
\[A/\rho=\{ a\rho:a\in A\}\] 
is a right $S$-act under $(a\rho)s=(as)\rho$, for any $a\in A$ and $s\in S$. 
The congruences of $S$ considered as a right $S$-act are precisely the right congruences of $S$ considered as a monoid.
For any right $S$-act $A$, and any $a\in A$, the annihilator of $a$ is
\[\ann(a):=\bigl\{ (u,v)\in S\times S:au=av\bigr\};\]
this is a right congruence on $S$ and $S/\ann(a)\cong aS$.

Generating sets for congruences
play a key role in this paper.
For a set $H$ consisting of ordered pairs we let $H^{-1}:=\bigl\{(x,y)\::\: (y,x)\in H\bigr\}$. We then let
$\overline{H}=H\cup H^{-1}$, which is the symmetric closure of $H$.
This notation will be used throughout the paper without further comment.
If $H\subseteq A\times A$, where $A$ is an $S$-act, we denote by $\langle H\rangle$
the congruence on $A$ \emph{generated} by $H$; that is the least congruence on $A$ that contains $H$.
Note that $\langle H\rangle=\langle \overline{H}\rangle$, which enables us to replace a (finite) generating set by a symmetric (finite) generating set for any congruence.
When we want to emphasise the act within which we are generating a congruence, we will write $\langle H\rangle_A$ for $\langle H\rangle$.
The following is a standard description of $\langle H\rangle$:

\begin{Lem}[{\cite[{Lemma 1.4.37}]{kkm:2000}}] \label{lem:cong}
Let $S$ be a monoid, $A$ an $S$-act, and $H \subseteq S \times S$. 
Then  for $a,b\in A$ we have $(a, b)\in\langle H\rangle$ if and only if $a=b$ or there exist $(c_1,d_1),\ldots,(c_n,d_n)\in \overline{H}$ and $t_1,\ldots, t_n \in S$ such that
\begin{equation}\label{Hseq}
a=c_1t_1,d_1t_1=c_2t_2,\ldots,d_nt_n=b.
\end{equation}
\end{Lem}

\begin{Def}
A sequence of the form \eqref{Hseq} is called  an \emph{$\overline{H}$-sequence} 
connecting $a$ and $b$. 
\end{Def}

\begin{Def} A right $S$-act $A$ over a monoid $S$ is {\em finitely presented} if it is isomorphic to $F_S(X)/\rho$ for some finite set $X$ and finitely generated congruence $\rho$.
\end{Def}

It is well-known from universal algebra that the notion of finite presentability does not depend on the choice of finite generating set; see \cite[Section 1.5 and Theorem 2.5.5]{burris}.
 In particular, any finitely presented monogenic (single-generator) $S$-act is of the form $S/\rho$ where $\rho$ is a finitely generated right congruence on $S$. 
Having set up all the necessary notations, we now give some equivalent characterisations of right coherent monoids that will be used interchangeably throughout the paper:

\begin{Thm}[{\cite{gould:1987,gould:1992}}]\label{thm:chasetype} 
The following are equivalent for a monoid $S$:
\begin{thmenumarab}
\item \label{it:coh1}
$S$ is right coherent;
\item \label{it:coh2}
for any finitely generated right congruence $\rho$ on $S$ and any elements $a,b \in S$:
\begin{thmenumrom}
\item \label{it:coh2i}
the subact $(a\rho) S \cap (b\rho) S$ of the right $S$-act $S/\rho$ is  finitely generated;
\item \label{it:coh2ii}
the annihilator 
\[
\ann(a\rho)=\bigl\{(u,v): au \ \rho\ av\bigr\}
\]
is a finitely generated right congruence on $S$;
\end{thmenumrom}
\item \label{it:coh3}
for any finite $X$ and finitely generated right congruence $\rho$ on $F_S(X)$ and any  $a,b \in F_S(X)$: 
\begin{thmenumrom}
\item  \label{it:coh3i}
the subact $(a\rho) S \cap (b\rho) S$ of the right $S$-act $F_S(X)/\rho$ is  finitely generated;
\item \label{it:coh3ii}
the annihilator
\[
\ann(a\rho)=\bigl\{(u,v): au \ \rho\ av\bigr\}
\]
is a finitely generated right congruence on $S$.
\end{thmenumrom}
\end{thmenumarab}
\end{Thm}

The equivalence of (1) and (2) in Theorem~\ref{thm:chasetype} is  \cite[Corollary 3.4]{gould:1992}; that of (2) and (3) follows from \cite[Theorem 6]{gould:1987} and
\cite[Proposition 3.2]{gould:1992}.

Let us briefly review some basic concepts of semigroup theory which will be used throughout.
For a monoid $S$ we will denote by $E(S)$ the set of all idempotents of $S$.
There is a natural partial order on $E(S)$ defined by $e\leq f$ if and only if $ef=fe=e$.
Green's relation $\ar$ is defined as follows:
for $a,b\in S$ we have $a\,\ar\, b$ if and only if $aS=bS$, or, equivalently, $a=bs$, $b=at$ for some $s,t\in S$. 
The relation $\el$ is defined dually,
 $\eh=\ar\cap\el$ and $\dee=\ar\circ\el=\el\circ \ar$. Finally, $a\,\jay\,b$ if and only if $SaS=SbS$.
A monoid is regular if and only if every $\ar$-class contains an idempotent, or equivalently, every $\el$-class contains an idempotent. If there is a {\em unique} idempotent in each $\ar$-class and in each $\el$-class then the monoid is {\em inverse}. 

Associated to $\jay$ is a preorder $\leq_{{\mathcal{J}}}$
defined by $a\leq_{\mathcal{J}}b$ if and only if $SaS\subseteq SbS$.
Clearly, $a\,\jay\, b$ if and only if $a\leq_{{\mathcal{J}}}b$ and $b\leq_{{\mathcal{J}}}a$.
More details about Green's relations and of other semigroup notions may be found in \cite{howie:1995}.

In this paper we will consider some standard semigroup constructions, built over a monoid $M$. 
For convenience we  now briefly introduce them; further details may be found in \cite{howie:1995}. 

Let $M$ be a monoid, $I,\Lambda$ non-empty sets, and $P=(p_{\lambda i})$ a $\Lambda\times I$ matrix with entries from $M$.
Then the \emph{Rees matrix semigroup} $\mathcal{M}(M;I,\Lambda;P)$
{over $M$
 is the set
$I\times M\times \Lambda$
with multiplication 
\[(i,a,\lambda)(j,b,\mu)=(i,ap_{\lambda j} b,\mu).\]
When}
 $M$ is a group $\mathcal{M}(M;I,\Lambda;P)$ is a completely simple semigroup 
(i.e. a semigroup with no proper two-sided ideals possessing minimal left and right ideals),
and all such semigroups arise in this way; see \cite[Theorem 3.3.1]{howie:1995}.

This construction can be generalised to  a \emph{Rees matrix semigroup with zero}
$\mathcal{M}^0(M;I,\Lambda;P)$ as follows: this time the entries of $P$ come from $M\cup\{0\}$, where $0$ is a new symbol not in $P$. The set of elements of $\mathcal{M}^0(M;I,\Lambda;P)$ is
$(I\times M\times \Lambda)\cup\{0\}$, and the multiplication is given by:
\begin{align*}
&(i,a,\lambda)(j,b,\mu)=
\begin{cases}
(i,ap_{\lambda j} b,\mu) & \text{if } p_{\lambda j}\in M,\\
0 & \text{if } p_{\lambda j}=0,
\end{cases}
\\
&
0(i,a,\lambda)=(i,a,\lambda)0=0\cdot 0=0.
\end{align*}
Rees matrix semigroups with $0$ over a group with the additional condition that every row and every column of $P$ contain at least one entry over $M$ are precisely completely $0$-simple semigroups, i.e. semigroups with $0$, with no proper non-zero ideals, and possessing minimal non-zero left and right ideals.

A special case of the latter construction is obtained by taking $I=\Lambda$, and $P$ the $I\times I$ identity matrix. This yields a \emph{Brandt semigroup} $\mathcal{B}(M;I)$.
Here the {the set of elements is}
$(I\times M\times I)\cup \{ 0\}$
and the multiplication is given by
\begin{align*}
&(i,a,j)(j,b,\ell)=
\begin{cases}
(i,ab,\ell) & \text{if } j=k,\\
0 & \text{if } j\neq k,
\end{cases}
\\
&
0(i,a,j)=(i,a,j)0=0\cdot 0=0.
\end{align*}
Brandt semigroups over groups are precisely completely $0$-simple inverse semigroups.

Rees matrix semigroups, or Brandt semigroups, are not (except in degenerate cases) monoids; if we adjoin an identity we refer to them as {\em Rees matrix monoids}  or {\em Brandt monoids}, respectively. 

The final construction we introduce here is that of a  \emph{Bruck--Reilly extension $\BR(M,\theta)$  of a monoid $M$}.
Here, $M$ is any monoid, $\theta:M\rightarrow M$ an endomorphism,
the {set of elements is} 
 $\N_0\times M\times \N_0$, where $\N_0$ denotes the natural numbers {\em including} zero,
and multiplication is given by
\[(a,g,b)(c,h,d)=(a-b+t, (g\theta^{t-b})(h\theta^{t-c}),d-c+t)
\quad \text{where}\quad t=\max\{ b,c\}. 
\]
Special instances of this construction were introduced by Bruck \cite{bruck:1958} and Reilly \cite{reilly:1966}, after whom it is now named. In the case where the image of $\theta$ is contained in the group of units of $M$, properties such as regularity pass between $M$ and $\BR(M,\theta)$; making such a restriction here does not advance our results and so we do not impose it.

Bruck--Reilly {extensions of groups} are precisely bisimple inverse $\omega$-semigroups,
i.e. inverse semigroups with a single $\dee$-class in which the natural order on idempotents is isomorphic to the inversely ordered chain $\omega$, i.e. $0>1>2>\dots$.

\section{Coherency and noetherian properties}
\label{sec:noeth}

We have commented that, in general, a weakly right noetherian monoid need not be right coherent. 
However, this implication \emph{does} hold for regular monoids. In fact, we can prove a more general result, building on
\cite[Theorem 5.2]{gould:2007},  where Gould showed that any regular  monoid in which every right ideal is principal is right coherent. 

We make repeated use of the following straightforward observation.

\begin{Lem}\label{lem:annihilators}\cite[Lemma 5.1]{gould:2007}
Let $I$ be a right ideal 
of a monoid $S$, let $\rho$ be a right congruence on $S$ and let $x,y\in S$.

\begin{thmenumrom}
\item \label{it:Ii}
The set
\[I\rho=\{ s\in S:s\,\rho\, a\mbox{ for some }a\in I\}=\bigcup
\{ a\rho:a\in I\}\]
is a right ideal of $S$ containing $I$. 
\item \label{it:Iii}
The set
\[(I,x)=\{ t\in S:xt\in I\}\]
is  a right ideal of $S$.
\item \label{it:Iiii}
If $x\,\rho\, y$ then $(I\rho,x)=(I\rho,y)$.
\end{thmenumrom}
\end{Lem}

The right ideal $I\rho$ defined in \ref{it:Ii} above is known as the {\em $\rho$-closure} of  $I$. In the following result, the reader should keep in mind the following notational subtlety: for an element $a\in S$, 
$(aS)\rho$ is a right ideal of $S$, whereas $(a\rho)S$ is the monogenic subact of $S/\rho$ generated by $a\rho$.

\begin{Thm}\label{prop:wrnregular}
If $S$ is a regular monoid in which for every finitely generated right congruence $\rho$ and every $a,b,x,y \in S$ the right ideals $(aS)\rho \cap (bS)\rho$ and $(aS,x) \cap (bS,y)$ are  finitely generated,
then $S$ is right coherent.
\end{Thm}

\begin{proof}
Let
$\rho$  be a finitely generated right congruence on $S$, where $S$ satisfies the conditions of the hypothesis. We may assume that
$\rho=\langle X\rangle$  for some finite symmetric subset $X\subseteq S\times S$.

We show that the conditions of Theorem~\ref{thm:chasetype}\ref{it:coh2} hold. To see that \ref{it:coh2}\ref{it:coh2i}  holds, let 
 $a,b \in S$.
Note that the map $\iota \colon (aS)\rho \cap (bS)\rho \to (a\rho) S \cap (b\rho) S$
such that $ x \mapsto x\rho$  is 
a surjective morphism between right $S$-acts.
Since the right ideal $(aS)\rho \cap (bS)\rho$ is finitely generated by assumption, so also is its image $(a\rho) S \cap (b\rho) S$.

Still with $a\in S$ we now show that the right congruence $\ann(a\rho)$ is  finitely generated, so that
the condition from Theorem~\ref{thm:chasetype}\ref{it:coh2}\ref{it:coh2ii} holds. We proceed by identifying  a  finite subset $Y$ of $\ann(a\rho)$, and then by showing that $Y$ is indeed a generating set.

By assumption the right ideal $(aS)\rho$ is finitely generated, so that, as $S$ is regular, 
\[(aS)\rho=\bigcup_{e\in K}eS\]
for some finite non-empty set $K\subseteq E(S)$.  For each $e\in K$ choose
$z_e\in S$ with $e\,\rho\, az_e$. Since $a\in (aS)\rho$ so we may choose and fix $f\in K$ such that  $a\in fS$, so that $a=fa\,\rho\, az_fa$, giving that
$(1,z_fa)\in \ann(a\rho)$.

For $p,q\in K$ and $x,y\in S$ we have by assumption that
$(pS,x)\cap (qS,y)$ is a finitely generated right ideal. Hence there exists a finite
(possibly empty) set $L(p,q,x,y)\subseteq E(S)$ with
\[
(pS,x)\cap (qS,y)=\bigcup_{h\in L(p,q,x,y)} hS.
\]
Now let
\begin{multline*}
Y:= \bigl\{ (1,z_fa)\bigr\}\cup \bigl\{ (z_pxh,z_qyh)\in \ann(a\rho) \::\: p,q\in K,\ (x,y)\in X,\\ 
h\in L(p,p,x,x)\cup L(p,q,x,y)\bigr\}.
\end{multline*}
The set $Y$ is finite and $Y\subseteq \ann(a\rho)$.  Let $\tau=\langle Y\rangle$ so that 
 $\tau\subseteq \ann(a\rho)$;   we need to prove the reverse inclusion.
To this end we state and verify two claims.

\begin{Claim}
\label{cl:nr1}
Suppose $(c,d)\in X$, $p\in K$ and $t\in S$ with $pct=ct$.
Then there exists $q\in K$ such that $(z_pct,z_qdt)\in\tau$ and $qdt=dt$.
\end{Claim}

\begin{proof}
From $pct=ct$ we have $t\in (pS,c)$, and hence there exists
$h\in L(p,p,c,c)$ such that $ht=t$.
Also we have $pch=ch$, and so
\begin{equation}
\label{eq:nr1}
az_pch \ \rho\ pch=ch\ \rho\ dh.
\end{equation}
It follows that $dh\in (aS)\rho$, and so there exists $q\in K$ with $qdh=dh$. Now
\begin{equation}
\label{eq:nr2}
dh=qdh\ \rho\ az_qdh.
\end{equation}
Together with \eqref{eq:nr1} this yields $(az_pch,az_qdh)\in\rho$ and so
$(z_pch,z_qdh)\in \ann(a\rho)$, giving 
$(z_pch,z_qdh)\in Y$ by definition. 
Now, to verify the claim, we have
\[
z_pct=z_pcht\ \tau\ z_q dht=z_qdt
\quad \text{and}\quad
qdt=qdht=dht=dt,
\]
as required.
\end{proof}

\begin{Claim}
\label{cl:nr2}
Suppose $(c,d)\in X$, $p\in K$ and $t,w\in S$. 
If $pct=ct$ and $dt=aw$ then
$(z_pct,w)\in\tau$.
\end{Claim}

\begin{proof}
Note that $fdt=faw=aw=dt$, implying
$t\in (pS,c)\cap (fS,d)$,
and hence $ht=t$ for some $h\in L(p,f,c,d)$.
Now, 
\[
az_p ch\ \rho\ pch=ch\ \rho\ dh=fdh\ \rho\ az_fdh,
\]
implying $(z_pch,z_fdh)\in \ann(a\rho)$, and hence
$(z_pch,z_fdh)\in Y\subseteq \tau$. But then
\[
z_pct=z_pcht\ \tau\ z_fdht=z_fdt=z_faw\ \tau\ w,
\]
as claimed.
\end{proof}

Returning to the main proof, suppose $(u,v)\in \ann(a\rho)$, i.e. $(au,av)\in\rho$.
Then either $au=av$, or else there is a sequence
\begin{equation}
\label{eq:nr3}
au=c_1t_1,d_1t_1=c_2t_2,\dots, d_nt_n=av,
\end{equation}
where for $1\leq j\leq n$ we have $(c_j,d_j)\in X$ and $t_j\in S$.
In the first case we have
\[
u=1u\ \tau\ z_fau=z_fav\ \tau\ 1v=v.
\]

Consider now the second case. 
First note that $fc_1t_1=fau=au=c_1t_1$. Thus, repeatedly applying Claim \ref{cl:nr1} along the sequence
\eqref{eq:nr3}, yields idempotents $f=p_0,p_1,\dots, p_n\in K$ such that
\begin{equation}
\label{eq:nr4}
z_{p_0}c_1t_1\ \tau\ z_{p_1}d_1t_1=z_{p_1}c_2t_2\ \tau \dots \tau\ z_{p_n}d_nt_n
\end{equation}
and 
\[
p_id_i t_i=d_it_i \quad\text{for all}\quad 1\leq i\leq n.
\]
Furthermore, using Claim \ref{cl:nr2} with $(c,d)=(c_n,d_n)$, $p=p_{n-1}$, $t=t_n$ and $w=v$, yields
$v\ \tau\ z_{p_{n-1}}c_nt_n\ \tau\ z_{p_n}d_nt_n$.
Similarly, Claim \ref{cl:nr2} with $(c,d)=(d_1,c_1)$, $p=p_1$, $t=t_1$ and $w=u$, gives
$u\ \tau\  z_{p_1}d_1t_1\ \tau\ z_{p_0}c_1t_1$. Combining this with \eqref{eq:nr4} yields $(u,v)\in\tau$,
as required.
\end{proof}

\begin{Cor}\label{cor:wrncoherent}
Every weakly right noetherian regular monoid is right coherent.
\end{Cor}

Examples of weakly right noetherian regular monoids include 
 completely regular monoids which can be expressed as a weakly (right) noetherian semilattice $Y$  of completely simple semigroups $S_{\alpha},\alpha\in Y$, where, additionally, 
for each $\alpha\in Y$ we have that $S_{\alpha}$ has finitely many $\ar$-classes.
They also include
 any regular $\omega$-semigroup, that is, any regular semigroup whose semilattice of idempotents is isomorphic to an inversely well-ordered chain $\omega$, i.e. $0>1>2>\dots$. If $S$ is of this kind, then it is well known that any right (and, indeed, left) ideal of $S$ is principal, so that, by Theorem~\ref{prop:wrnregular}, $S$ must be coherent. 
As we mentioned in Section \ref{sec:prelim}, bisimple inverse $\omega$-semigroups are precisely the Bruck--Reilly extensions of groups, and so we have:

\begin{Cor}
\label{cor:BRG}
Every Bruck--Reilly extension $\BR(G,\theta)$ {of} a group $G$ is coherent.
\end{Cor}

The conditions of Theorem \ref{prop:wrnregular} are in fact weaker than the condition of being weakly right noetherian, as we now show via two examples. The first concerns a Rees matrix semigroup $\mathcal{M}=\mathcal{M}(G;I,\Lambda;P)$ over a group $G$ (i.e. {a completely simple semigroup}). It is easy to see that for any $i\in I$, the set
$R_i=\{ (i,g,\lambda):g\in G,\lambda\in\Lambda\}$ is a right ideal, and is generated by any of its elements. Thus any right ideal $K$ of $\mathcal{M}$ is of the form $K=\bigcup_{j\in J}R_j$ for some $J\subseteq I$.

\begin{Prop}\label{prop:cs1} 
Let $S=\mathcal{M}(G;I,\Lambda;P)^1$ be a Rees matrix monoid over a group $G$. Then $S$ satisfies the conditions of Theorem~\ref{prop:wrnregular} and hence is right coherent. However, $S$ is weakly right noetherian if and only if $I$ is finite. 
\end{Prop}

\begin{proof} We begin by noting that the proper right ideals of $S$ are exactly the right ideals of $\mathcal{M}$, and
thus of the form
$K=\bigcup_{j\in J}R_j$ where $J\subseteq I$. Moreover, $K$ is finitely generated if and only if $J$ is finite.
The final statement of the proposition follows immediately. We also observe that  if  $K$ and $L$ are finitely generated right ideals of $S$, then so is $K\cap L$. Thus to show that the conditions of Theorem~\ref{prop:wrnregular} hold, it is sufficient to show that for any finitely generated right congruence $\rho$ on $S$ and any $\a,\b\in S$ the right ideals $(\a S)\rho$  and  $(\a S,\b )$ are finitely generated. 

Let $X$ be a  finite symmetric set of generators for $\rho$. Suppose first that $1\rho=\{ 1\}$, so that  $X$ consists entirely of pairs of the form $\big((i,g,\lambda), (j,h,\mu)\big)$. 
If $\a=1$ then  $(\a S)\rho=S=1S$. Otherwise, let $\a=(\ell, k,\nu)$
so that $\a S=\{ \ell\} \times G\times \Lambda$. If $(o, x, \kappa)\,\rho\, (\ell, y, \tau)$
for some $(\ell, y,\tau)$ then either $(o,  x, \kappa)=(\ell, y, \tau)$  or
$(o, x, \kappa)= (i,g,\lambda){\mathbf{t}}$ for some $(i,g,\lambda)$ with $\big((i,g,\lambda), (j,h,\mu)\big)\in X$ and ${\mathbf{t}}\in S$. In the first case  $o=\ell$ and in  the second $o=i$. Since $X$ is finite we deduce that $(\a S)\rho$ is finitely generated. 

Suppose now that $1\rho\neq \{ 1\}$, so that $X$ contains at least one pair of the form
$\big((i,g,\lambda), 1\big)$. Again, if $\a=1$ then $(\a S)\rho=S$. On the other hand, if  $\a =(\ell, k,\nu)$ then
\[1\ \rho\ (i,g,\lambda)=(i,g,\lambda)(\ell,p_{\lambda \ell}^{-1},\lambda)\ \rho\ 
1(\ell,p_{\lambda \ell}^{-1},\lambda)=(\ell ,p_{\lambda \ell}^{-1},\lambda),\]
so that $1\in \big((\ell,p_{\lambda \ell}^{-1},\lambda)S\big)\rho=(\a S)\rho$,
giving that $(\a S)\rho=S$.

Now we have:
\begin{txitemize}
\item
if $\a=1$ then $(\a S,\b)=S$;
\item
 if $\b=1$ then
$(\a S,\b)=\a S$;
\item
if $\a,\b\neq 1$ and $\a,\b\in R_i$ for some $i\in I$, then $(\a S,\b)=S$;
\item
if $\a,\b\neq 1$ with $\a\in R_i$, $\b\in R_j$ where $i\neq j$, then  $(\a S,\b)=\emptyset$.
\end{txitemize}
In any case $(\a S,\b)$ is finitely generated, which concludes the proof.
\end{proof}

We end this section by giving two examples of regular monoids that are right coherent but are  not weakly right noetherian. 
The first of them is a Brandt monoid over a group, and {it} does not even satisfy the conditions of Theorem~\ref{prop:wrnregular}.

\begin{Prop}\label{prop:brandtinf} 
Let $S=\mathcal{B}(G;I)^1$, where $G$ is a group and $I$ is infinite, and let 
$\rho=\bigl\langle \big( (i,g,i),1\big)\bigr\rangle $ for fixed $i\in I$, $g\in G$.
Then $S$ is right coherent, but the ideal $(0S)\rho$ is not finitely generated.
\end{Prop}

\begin{proof}
That $S$ is right coherent will be proved in Theorem \ref{thm:brandt}.
To prove the second assertion,
 let $J=\{ 0\} \cup \bigcup_{j\in I,j\neq i}R_j$, and we claim that $J=(0S)\rho$. For any $j\in I$ with $j\neq i$ we have 
\[(j,h,j)=1(j,h,j )\, \rho\,   (i,g,i)(j,h,j )=0,\]
and it
 follows that the right ideal $J\subseteq (0S)\rho=0\rho$.
We claim that $1$, or equivalently $(i,g,i)$,  cannot be in $(0S)\rho$. Recall that
since $R_i$ is a right ideal, we have $(i,g,i)\in (0S)\rho$ if and only if 
$(i,h,\lambda)\in (0S)\rho$ for any $(i,h,\lambda)\in R_i$. But, the latter is impossible, as we now show. 
Suppose that $(i,h,\lambda)\in R_i$ and $(i,h,\lambda)$ is related to
$\mathbf{s}\in S$ via an $X$-sequence of length 1. Then $(i,h,\lambda)=\mathbf{c}\mathbf{t}$ and $\mathbf{d}\mathbf{t}=
\mathbf{s}$,  where $(\mathbf{c},\mathbf{d})$ or $(\mathbf{d},\mathbf{c})$ is $ \big( (i,g,i),1\big)$ and
$\mathbf{t}\in S$. If $(\mathbf{c},\mathbf{d})= \big( (i,g,i),1\big)$,
then either $\mathbf{t}=1$ and then $\mathbf{s}=1$, or $\mathbf{t}=(i,g^{-1}h,\lambda)$ and then also
$\mathbf{s}= (i,g^{-1}h,\lambda)$. Similarly, if $(\mathbf{d},\mathbf{c})= \big( (i,g,i),1\big)$,
then $\mathbf{s}=(i,gh,\lambda)$. Since a single step in an $X$-sequence starting from $1$ can only take us to $(i,g,i)$, our claim follows by induction on the length of $X$-sequences.   
Consequently, 
 $(0S)\rho=J$.   As $I$ is infinite, 
$(0S)\rho$ is not finitely generated as a right ideal of $S$. \end{proof}

\smallskip 
For our second example, we return to a regular semigroup $S$ with $\omega$-chain of idempotents. 
If $S$ is bisimple, then, as commented in Section~\ref{sec:prelim}, $S$ is isomorphic to a Bruck--Reilly {extension} $\BR(G,\theta)$. 
This construction was extended by Warne \cite{warne:1968}, who noticed that one could replace $\N_0$ by $\Z$, and, using the same rule for multiplication, obtain an inverse monoid with chain of idempotents isomorphic to $\Z$ (where the natural correspondence inverts the order). This construction yields an
 {\it extended Bruck--Reilly} {extension} $\mathrm{EBR}(G,\theta)$. Note that
 $\mathrm{EBR}(G,\theta)$ is inverse, and is neither a monoid nor  weakly right noetherian. 
 
\begin{Prop}\label{prop:DBR}
Let $G$ be a group,  $\theta \colon G \to G$ a homomorphism, and $S=\mathrm{EBR}(G,\theta)^1$. Then $S$ is right coherent, but not weakly right noetherian.
\end{Prop}

\begin{proof}
Denote the identity of $G$ by $e$.  
It is entirely routine to check that the right ideals of $S$ are precisely $S,S\setminus\{1\},\emptyset$ and $R_i=\{(j,g,k):j\geq i\}$  for every $i\in \mathbb{Z}$, 
and that $S\setminus \{1\}$ is the only right ideal that is not finitely generated.

To prove that $S$ is right coherent we make use of
Theorem \ref{prop:wrnregular}.
To this end, let
$\rho$ be a right congruence on $S$ generated by the finite symmetric set $X \subseteq S \times S$.
Since the right ideals of $S$ are linearly ordered, it is enough to show that for any $\a\in S$ and $\mathbf{e}\in E(S)$ we can have
neither $(\a S)\rho$ nor $(\mathbf{e} S,\a)$ being equal to the non-finitely generated right ideal $S\setminus \{1\}$.

We first consider $(\a S)\rho$. If 
 $\{1\}$ is not a $\rho$-class then as $(\a S)\rho$ is a union of $\rho$-classes, {we have}
 $(\a S)\rho\neq S\setminus \{1\}$.
 On the other hand, if $\{1\}$ is a $\rho$-class then let $n$ be an integer smaller than any number appearing in any of the first  {coordinates} of elements of $X$.
Then $(n,e,n)\rho=\{(n,e,n)\}$, because $(n,e,n)$ cannot be written as $\mathbf{c}\mathbf{t}$ where $(\mathbf{c},\mathbf{d})\in X$ and $\mathbf{t}\in S$.
As a consequence, $(\a S)\rho \neq S\setminus \{1\}$ for any $\a\in S$. 

Consider now the right ideal  $(\mathbf{e}S,\a )$. If $\a\in \mathbf{e}S$ then $(\mathbf{e}S,\a)=S$ and if $\a=1$, then $(\mathbf{e}S,\a)=\mathbf{e}S$. Otherwise, $\mathbf{e}=(i,e,i)$ and $\a=(p,c,q)$ where $p<i$, and then an easy calculation yields $(\mathbf{e}S,\a)=(q+i-p,e,q+i-p)S$. 
\end{proof}

\section{Submonoids}
\label{sec:down}

In this section we begin the examination of when right coherency of a monoid $S$ is inherited by 
certain monoid subsemigroups. Recall that if $T$ is a subsemigroup of a monoid $S$ then
$T$ is a {\em retract} of $S$ if there is a morphism 
(called \emph{retraction}) $\theta:S\rightarrow S$ such that $\Im \theta=T$ and $\theta|_T:T\rightarrow T$ is the identity  map  on $T$.
 It is clear that if $T$ is a retract of $S$ then it {must be} a monoid subsemigroup. 

\begin{Prop}\cite[Theorem 6.3]{GM}\label{prop:retracts} Let $S$ be a right coherent monoid and let $T$ be a retract of $S$. Then $T$ is right coherent. 
\end{Prop}

\begin{Cor}\label{cor:ideal} Let $I$ be an ideal of a right  coherent monoid $S$ such that
$I$ has an identity. Then $I$ is right coherent.
\end{Cor}
\begin{proof} If $e$ is the identity of $I$, then it is easy to see that
$\theta:S\rightarrow I$ given by $a\theta=ea$ is a retraction onto $I$. The result now follows by Proposition~\ref{prop:retracts}.
\end{proof}

For a semigroup $M$, we denote by
$M^{\underline{1}}$ the monoid obtained by adjoining an identity $\underline{1}$ {\em whether or not $M$ already possesses one}. 
As another application of the above result, we obtain one half of the following:

\begin{Prop}\label{cor:1}  Let $M$ be a monoid. Then $M$ is right coherent if and only if 
$M^{\underline{1}}$ is right coherent.
\end{Prop}

\begin{proof}
In what follows we denote the identity of $M$ by $1_M$, and that of $M^{\underline{1}}$ by $\underline{1}$.
If $M^{\underline{1}}$ is right coherent, then so is $M$ by Corollary \ref{cor:ideal}, since it is a monoid ideal.

For the converse,
let $\rho$ be a finitely generated right congruence on $M^{\underline{1}}$, 
and let $\nu=\rho\cap (M\times M)$ be its restriction to $M$.
We claim that $\nu$ is finitely generated. Specifically, we prove that
if $\rho$ is generated by a finite set $H$ then
\[
\nu=\langle K\rangle_M \quad \text{where} \quad 
K=\bigl(H\cap (M\times M)\bigr)\cup \bigl\{ (1_M,a)\::\: a\in M,\ (\underline{1},a)\in \overline{H}\bigr\}.
\]
That $K\subseteq \nu$ and hence  $\langle K\rangle_M\subseteq \nu$ follows from the definition, and the observation that
 $(\underline{1},a)\in \overline{H}$ implies $(1_M,a)=(\underline{1}\, 1_M,a\, 1_M)\in \rho$.
 For the reverse inclusion, let 
 {$a,b\in M$ with  $a\neq b$ and $(a,b)\in\rho$.}
 Then there is an $H$-sequence
 \[
 a=c_1t_1,d_1t_1=c_2t_2,\dots, d_nt_n=b,
 \]
where $(c_i,d_i)\in\overline{H}$ and $t_i\in M^{\underline{1}}$ for $1\leq i\leq n$.
Multiplying this sequence on both sides by $1_M$ yields a $K$-sequence over $M$ from $a$ to $b$,
as required.

We now want to prove that
$\ann (a\rho)$ is finitely generated for any $a\in M^{\underline{1}}$. If $a=\underline{1}$ then
$\ann(a\rho)=\rho$ and there is nothing to prove.
So suppose $a\in M$. We claim that:
\begin{equation}
\label{eq:M1_2}
\ann(a\rho)=\kappa \quad \text{where}\quad \kappa=\bigl\langle\ann(a\nu)\cup\{ (1_M,\underline{1})\}\bigr\rangle_{M^{\underline{1}}}.
\end{equation}
{Since} $\ann(a\nu)\cup\{(1_M,\underline{1})\}\subseteq \ann(a\rho)$, so $\kappa\subseteq\ann(a\rho)$.
For the reverse inclusion, consider $(u,v)\in \ann(a\rho)$.
If $(u,v)\in M\times M$ then $au\,\rho\, av$, so $au\, \nu\, av$, and hence $(u,v)\in\ann(a\nu)\subseteq\kappa$.
Clearly $(\underline{1},\underline{1})$ belongs to both $\ann(a\rho)$ and $\kappa$.
Finally, if $u=\underline{1}$ and $v\in M$, then we have $a\underline{1}=a=a1_M\,\rho\, av$, so that $(1_M,v)\in\ann(a\nu)\subseteq \kappa$; together with $(\underline{1},1_M)\in\kappa$ this implies
$(\underline{1},v)\in\kappa$, as required.

Having established \eqref{eq:M1_2} it immediately follows that $\ann(a\rho)$ is finitely generated.
Indeed, since $M$ is right coherent by assumption, and $\nu$ is finitely generated,
it follows that $\ann(a\nu)$ is finitely generated as an $M$-act. Thus, any finite generating set for $\ann(a\nu)$ together
with $(1_M,\underline{1})$ will generate $\ann(a\rho)$ as an $M^{\underline{1}}$-act.

Now let $a,b\in M^{\underline{1}}$; we need to show that 
$(a\rho)M^{\underline{1}}\cap (b\rho)M^{\underline{1}}$ is finitely generated.
If $(a\rho)M^{\underline{1}}=(\underline{1}\rho)M^{\underline{1}}$ then
$(a\rho)M^{\underline{1}}\cap (b\rho)M^{\underline{1}}=(b\rho)M^{\underline{1}}$;
likewise if $(b\rho)M^{\underline{1}}=(\underline{1}\rho)M^{\underline{1}}$.
So suppose neither $(a\rho)M^{\underline{1}}$ nor $(b\rho)M^{\underline{1}}$ equals
$(\underline{1}\rho)M^{\underline{1}}$.
Then $a,b\in M$ and
\[
(a\rho)M^{\underline{1}}\cap (b\rho)M^{\underline{1}}=(a\nu)M\cap (b\nu)M.
\]
The right hand side is finitely generated as an $M$-act, because $M$ is right coherent by assumption and $\nu$ is finitely generated. Hence it is also finitely generated as an $M^{\underline{1}}$-act, completing the proof of the claim and the proposition.
\end{proof}

There is an analogue of Proposition~\ref{cor:1} for
$M^{\underline{0}}$, the monoid obtained by adjoining a new zero element to $M$,
but we will postpone it until the next section (Corollary \ref{cor:0}).

Corollary~\ref{cor:ideal} is about principal ideals possessing an identity. With some further work, we obtain a corresponding result for monoid $\jay$-classes. We first present
 a general result on extendability of right congruences on a submonoid.

\begin{Def} Let $T$ be a subsemigroup of a monoid $S$, and suppose that $T$ has identity $e$. We say that $T$ is:
\begin{thmenumrom}
\item \label{it:Unitaryi}   
{\em left unitary} if $a,ab\in T$ implies that $b\in T$;
\item \label{it:unitaryii}
{\em weakly left unitary} if $a,ab\in T$ implies that $eb\in T$. 
\end{thmenumrom}
The notions of being right unitary and weakly right unitary are dual.
\end{Def} 

If $T$ is left unitary, then it is weakly left unitary. For an immediate   counterexample to the converse, consider the bicyclic monoid $B$; {the subsemigroup} $\{ (1,1)\}$ is weakly left unitary but not left unitary.  

\begin{Def}\label{defn:SRCEP} A subsemigroup $T$ of a monoid $S$ has the {\em strong right congruence extension property SRCEP} if, for any right congruence $\rho$ on $T$,  we have \[\rho=\rho^S\cap (T\times T)\] 
and $T$ is a union of $\rho^S$-classes, where $\rho^S$ is the right congruence on $S$ generated by $\rho$.
\end{Def}

We note that $T$ has 
 SRCEP if and only if for any 
right congruence
$\rho$  on $T$ we have
$\rho=\tau\cap (T\times T)$, and $T$ is a union of $\tau$-classes, for {\em some} right congruence $\tau$ on $S$. 

\begin{Lem}\label{lem:leftunitary}
Let $S$ be a monoid and let $T$ be a  monoid subsemigroup with identity $e$. 
If $T$ is weakly left unitary then $T$ has SRCEP.
\end{Lem}
\begin{proof} 
Let $\rho$ be a right congruence on $T$. Clearly, $\rho\subseteq \rho^S\cap
(T\times T)$. Conversely, if $a\in T$ and $a\,\rho^S\, b$ where {$b\in S$ and $b\neq a$}, then  there
exists a $\rho$-sequence 
\[a=c_1t_1,\, d_1t_1=c_2t_2,\hdots ,d_nt_n=b\]
where $(c_i,d_i)\in \rho$ and $t_i\in S$ for $1\leq i\leq n$.
From $a=c_1t_1=(c_1e)t_1=c_1(et_1)$ and the fact $T$ is weakly left unitary, we have $et_1\in T$, so that
$d_1t_1=(d_1e)t_1=d_1(et_1)\in T$. In the same way we obtain $et_2,\hdots ,et_n\in T$, and it follows that $b\in T$ and
$a\,\rho\, b$ as required.
\end{proof}

\begin{Lem} \label{basic} Let $J$ be a monoid $\mathcal{J}$-class of $S$ with identity $e$.
Then $J$ is weakly left and right unitary.
\end{Lem}

\begin{proof} 
Let $a,ab \in J$.
Then $ab=aeb\, \leq_\mathcal{J}\, eb\,  \leq_\mathcal{J}\,  e\ \mathcal{J}\ ab$ implies that $eb\ \mathcal{J} \ ab$, that is, $eb \in J$, showing that $J$ is  weakly left unitary. 
Dually, $J$ is weakly right unitary. 
\end{proof}

\begin{Thm}\label{thm:jclasses} 
Let $S$ be a right coherent monoid  and let $J$ be a monoid $\mathcal{J}$-class of $S$. Then $J$ is 
 right coherent. 
 \end{Thm}

\begin{proof} From Lemmas~\ref{lem:leftunitary} and~\ref{basic}, we have that $J$ is weakly unitary and has SRCEP.
Let $e$ be the identity of $J$ and let  $\rho$ be a finitely generated right congruence on $J$. 

We will check that the conditions of Theorem~\ref{thm:chasetype}\ref{it:coh2} hold. 
First we show that $\ann(a\rho)$ is finitely generated for any $a\in J$, so
that Condition \ref{it:coh2}\ref{it:coh2ii} holds.
Since $S$ is right coherent and $\rho^S$ is finitely generated,  $\ann(a\rho^S)$ is generated by a finite set $H
\subseteq S\times S$. 
We claim that $\ann(a\rho)$ is generated by the finite set 
\[
H'=\bigl\{(ec,ed): (c,d) \in H \text{ and } ec,ed \in J\bigr\}.
\]
To see this, note first that for any $(c,d)\in H$ we have
\[
a(ec)=(ae)c=ac\ \rho^S\ ad=(ae)d=a(ed),
\]
and it follows from the fact that $J$ has SRCEP that $H'\subseteq \ann(a\rho)$. For the converse, suppose that 
$(u,v)\in\ann(a\rho)$ where $u\neq v$, so that $u,v\in J$ and $au\,\rho\, av$. Hence $au\,\rho^S\, av$ so that
$(u,v)\in\ann(a\rho^S)$ and there is an $H$-sequence
\begin{equation}\label{eqn:seq}u=c_1t_1,\, d_1t_1=c_2t_2,\, \hdots ,d_nt_n=v\end{equation}
where $(c_i,d_i)\in \overline{H}$ and $t_i\in S$, for $1\leq i\leq n$. Multiplying the sequence (\ref{eqn:seq}) on the left by $a$ we have
\[
au=ac_1t_1,\, ad_1t_1=ac_2t_2,\, \hdots ,ad_nt_n=av
\]
and it follows that $au\,\rho^S\, ac_it_i\,\rho^S\, ad_it_i\,\rho^S\, av$, for $1\leq i\leq n$. 
Since $J$ has SRCEP, we have $ac_it_i,ad_it_i\in J$ and then
$ec_i,ed_i,t_ie \in J$ for $1\leq i\leq n$, giving $(ec_i,ed_i)\in H'$. Returning to sequence
(\ref{eqn:seq}), multiplying through on both left and right by $e$, we obtain an $H'$-sequence connecting $u$ to $v$. 
Hence $\ann(a\rho)$ is generated by $H'$.

We now turn our attention to the Condition \ref{it:coh2}\ref{it:coh2ii} in Theorem~\ref{thm:chasetype}. At this stage it is helpful to notice that if $s,t\in S$ and $s\,\rho^S\, t$, then, since $\rho^S$ is generated by pairs of elements from $J$, we have $es\,\rho^S\, et$ and, if $s\neq t$, then $es=s$ and $et=t$.

 Let
 $a,b \in J$ be such that $(a\rho) J \cap (b\rho) J \neq \emptyset$. Consequently,
 $au\,\rho\, av$ for some $u,v\in J$ so that $au\,\rho^S\, bv$ and
 $(a\rho^S)S  \cap (b\rho^S) S \neq \emptyset$. Since $S$ is right coherent we have that 
\[(a\rho^S)S  \cap (b\rho^S) S = YS\] for some finite $Y \subseteq S/\rho^S$.
Let $X$ be a transversal of the $\rho^S$-classes in $Y$,
 let $X'=eX \cap J$ and put  $Y'=\{ x'\rho:x'\in X'\}$.
We claim that \[(a\rho) J \cap (b\rho) J=Y' J.\]

Let $x'\rho\in Y'$ where $x'=ex\in J$ and $x\in X$. Then $x\rho^S\in Y$ so there are
$p,q\in S$ with $ap\,\rho^S\, bq\, \rho^S\, x$. It follows by an earlier remark that $ap\,\rho^S\, bq\, \rho^S\, ex=x'(=x)\in J$
and by SRCEP, $ap,bq\in J$ and $a(ep)=ap\,\rho\, bq=b(eq)\, \rho\, x'$. Further,
$ep,eq\in J$ since $J$ is weakly left unitary. This shows that 
$Y'\subseteq (a\rho)J\cap (b\rho)J$. 

Conversely, suppose that $c\rho\in (a\rho)J\cap (b\rho)J$ so that $c\,\rho\, ah\,\rho\, bk$  for some $h,k\in J$. 
Then $c\rho^S\in (a\rho^S)S\cap (b\rho^S)S$ so 
that $c\rho^S=(x\rho^S)w$ for some $x\rho^S\in Y$, where we can take $x\in X$, and $w\in S$. It follows that
$ah\,\rho^S\, bk\, \rho^S\, xw$, giving that $xw\in J$ and $ah\,\rho\, bk\, \rho \, xw$, since $J$ satisfies the SRCEP. 
Now $xw=exw\leq_{\mathcal{J}}ex\leq_{\mathcal{J}}e$, so that $ex\in J$ and hence
$ex\in X'$ and $(ex)\rho\in Y'$. Finally, $xw=exw$ so that as $J$ is weakly left unitary and $ex\in J$, we have $ew\in J$
and $xw=e(xw)=(ex)w=(ex)(ew)$,   and  we conclude $c\rho=\big((ex)\rho \big)(ew)\in Y' J$. Thus
$(a\rho)J\cap (b\rho)J=Y'J$, as required.
\end{proof}

Let $M$ be a monoid. It is easy to see that $M$ is a monoid subsemigroup of any Brandt monoid $\mathcal{B}(M,I)^1$ or Bruck--Reilly {extension} $\BR(M,\theta)$ built over $M$. In the remainder of this section we set up machinery to show that if either  $\mathcal{B}(M,I)^1$ or $\BR(M,\theta)$ is right coherent, then so is $M$. 
Though our approach is quite technical, it allows us to treat these two cases, and others,  together. 

We will make use of the relations $\arte, \elte$ and $\ehte$ on a monoid $S$, 
introduced in \cite{L1}, where
$E\subseteq E(S)$ is a set of idempotents.  For any $a,b\in S$ we have that $a\,\arte\, b$ if and only if for all $e\in E$, 
\[
ea=a\Leftrightarrow eb=b.
\]
The relation $\elte$ is defined dually, and we put $\ehte=\arte\cap\elte$.  Clearly, $\arte,\elte$ and $\ehte$ are equivalence relations on $S$. It is easy to see that $\ar\subseteq \arte, \el\subseteq \elte$ and $\eh\subseteq \ehte$, and  if $S$ is regular and $E=E(S)$, then $\ar= \arte, \el=\elte$ and so $\eh=\ehte$. In this vein, if $u\,\arte\, v$ and $uS\subseteq wS\subseteq vS$ for some $w\in S$, then it is easy to see that $u\,\arte\, w\,\arte\, v$.

Suppose now that  $M$ is a monoid subsemigroup of a monoid $S$, with  identity $e$, and that the following conditions hold:  

\begin{thmenumalph}
\item \label{it:tila}
$e\in E$ for some $E\subseteq E(S)$;
\item \label{it:tilb}
$\eht_E$ is a right congruence;
\item \label{it:tilc}
$M$ is the $\eht_E$-class  of $e$; 
\item \label{it:tild}
for any $u,v\in S$, if $e\,\art_E\, u\, \eht_E\, v$, then $u=u(pq), v=v(pq)$ for some
$p,q\in S$ with $up,vp\in M$.
\end{thmenumalph}

Note that \ref{it:tild} would certainly hold if $S$ were regular and $E=E(S)$, 
from \cite[Lemmas 2.2.1 and 2.2.2]{howie:1995}. 
In this case $M$ would be a group, which we know to be right (and left) coherent.

\begin{Prop}\label{prop:thetechresult} Let $M$ be a monoid subsemigroup of a right coherent monoid $S$, such that
$M$ has identity $e$ and \ref{it:tila}--\,\ref{it:tild} above hold. Then $M$ is right coherent.
\end{Prop}

\begin{proof}  First note that if $a,ab\in M$, then as $e\,\ehte\, a$ and
$\ehte$ is a right congruence on $S$, we have $eb\,\eht_E\, ab\in M$, so that $M$ is weakly left unitary. By Lemma~\ref{lem:leftunitary} $M$  has SRCEP.   

Suppose now that $\rho$ is a finitely generated right congruence on $M$; say
$\rho=\langle H\rangle_M$, where $H\subseteq M\times M$ is finite. {Then} $\rho^S:=\langle H\rangle_S$ is finitely generated as a right congruence on $S$. Notice that as $\eht_E$ is also a right congruence on $S$, and
$H\subseteq \eht_E$, we have that $\rho^S\subseteq \eht_E$.

Let $a\in M$; we show that $\ann(a\rho)$ is finitely generated. By assumption that $S$ is right coherent, we have
$\ann(a\rho^S)=\langle K\rangle_S$ for some finite  $K\subseteq S\times S$.

For any $(u,v)\in K$ we have by definition that $au\,\rho^S\, av$, 
so that as $\rho_S\subseteq \eht_E$ and $\ehte$ is  a right congruence, it follows that
$eu\,\eht_E\, ev$. For any pair $(u,v)\in K$ with $e\,\art_E\, eu\,\eht_E\, ev$, choose elements
$p_{(u,v)},q_{(u,v)}\in S$ guaranteed by \ref{it:tild} such that 
\[
eu=eu(p_{(u,v)}q_{(u,v)}),\quad
ev=ev(p_{(u,v)}q_{(u,v)})\quad \text{and}\quad eup_{(u,v)}, evp_{(u,v)}\in M.
\]
 Notice that  
\[aeup_{(u,v)}=aup_{(u,v)}\,\rho^S\, avp_{(u,v)}=aevp_{(u,v)},\]
so that as 
$M$ has SRCEP we have 
\[K'=\bigl\{ (eup_{(u,v)},evp_{(u,v)}):(u,v)\in K, e\,\art_E\, eu\,\eht_E\, ev\bigr\}\]
is a finite subset of $\ann(a\rho)$. 

We claim that $K'$ generates $\ann(a\rho)$. To see this, suppose that $(h,k)\in \ann(a\rho)$, where $h\neq k$, so that
$ah\,\rho^S\, ak$ and there exists a $K$-sequence 
\[h=u_1t_1,\, v_1t_1=u_2t_2,\hdots ,v_nt_n=k\]
where $(u_i,v_i)\in \overline{K}$ and $ t_i\in S$ for $1\leq i\leq n$.
Consequently, 
\[
h=eh=eu_1t_1,\, ev_1t_1=eu_2t_2,\hdots ,ev_nt_n=ek=k.
\]
We have $hS\subseteq eu_1S\subseteq eS$, so by an earlier remark, 
 $e\,\art_E\, eu_1$, and we have already noted that  $eu_1 \,\eht_E\, ev_1$. Consequently, there exist
$p_{(u_1,v_1)}, q_{(u_1,v_1)}$ with 
\[
(eu_1p_{(u_1,v_1)}, ev_1p_{(u_1,v_1)})\in K',\quad
eu_1p_{(u_1,v_1)}q_{(u_1,v_1)}=eu_1\quad \mbox{and}\quad ev_1p_{(u_1,v_1)}q_{(u_1,v_1)}=ev_1.
\] 
Thus
\[h=eu_1t_1=eu_1p_{(u_1,v_1)}q_{(u_1,v_1)}t_1=(eu_1p_{(u_1,v_1)})(eq_{(u_1,v_1)}t_1).\]
Since $M$ is weakly left unitary and $eu_1p_{(u_1,v_1)}\in M$, we deduce that $ eq_{(u_1,v_1)}t_1$ lies in $M$ and
\[h=(eu_1p_{(u_1,v_1)})(eq_{(u_1,v_1)}t_1)\, \langle K'\rangle_M\,  (ev_1p_{(u_1,v_1)})(eq_{(u_1,v_1)}t_1)=ev_1t_1
=eu_2t_2\in M.\]
Continuing in this manner we obtain $h\, \langle K'\rangle_M\, ev_nt_n=k$.
Thus $\langle K'\rangle_M=\ann(a\rho)$,  as required.

Next we show that $(a\rho)M\cap (b\rho)M$ is finitely generated for any $a,b\in M$. 
Since $S$ is right coherent there is a  finite set $C\subseteq S$ such that
\[(a\rho^S)S\cap (b\rho^S)S=\bigcup_{c\in C}(c\rho^S)S.\]
For any $c\in C$ we have $c\,\rho^S\, au\,\rho^S\, bv$ for some
$u,v\in S$, whence $c\,\eht_E\, au\, \eht_E\, bv$. Since $ea=a$ we  have $ec=c$. If $c\,\art_E\, e$, then we choose $p_c,q_c$ such that $c=cp_cq_c$ and $cp_c\in M$. Notice that $cp_c\,\rho^S\, aup_c\,\rho^S\, bvp_c$ so that by Lemma~\ref{lem:leftunitary} as $M$ has 
SRCEP,
$cp_c\,\rho\, a(eup_c)\,\rho\, b(evp_c)$, and
$eup_c,evp_c\in M$ as $M$ is weakly
 left unitary.
Thus
\[T\subseteq (a\rho)M\cap (b\rho)M,\]
where
\[T=\{ (cp_c)\rho:c\in C, c\,\art_E\, e\}.\]
Conversely, if $w\in M$ and $w\,\rho\, ah\,\rho\, bk$ for some $h,k\in M$, then 
we have $w\,\rho^S\, c\ell$ for some $c\in C$ and $\ell\in S$. As $M$ has SRCEP,
$c\ell\in M$ and from $c\ell S\subseteq cS\subseteq eS$ we have that
$c\,\arte\, e$.  Choosing  $p_c,q_c$ as above we have $(cp_c)\rho\in T$ and
\[w\,\rho\, c\ell=cp_cq_c\ell=(cp_c)(eq_c\ell).\]
Weak left unitariness of $M$ gives that
$eq_c\ell\in M$. Thus  $w\rho\in \big((cp_c)\rho\big) M$. It follows that $T$ generates
$(a\rho)M\cap (b\rho)M$, as required.
\end{proof}

\begin{Cor}
\label{cor:btdown}
Let $S=\mathcal{B}(M;I)^1$ be a Brandt monoid over a monoid $M$. If $S$ is right coherent, then so is $M$.
\end{Cor}

\begin{proof}
 Let $1_M$ be the identity of $M$ and put
$E=\{ (i,1_M,i):i\in I\}\cup\{ 0\}$. It is easy to see that $\{1\}$ and $\{0\}$ are $\art_E$ and $\elt_E$-classes, and  for any $(i,a,j),(k,b,\ell)$ we have that 
\[(i,a,j)\,\art_E\, (k,b,\ell)\Leftrightarrow i=k\]
and
\[(i,a,j)\,\elt_E\, (k,b,\ell)\Leftrightarrow j=\ell.\]
{The relation} $\eht_E$ is a congruence on $S$. Fixing $i\in I$ and putting
$M_i=\{ i\} \times M\times \{ i\}$, we see that \ref{it:tila}, \ref{it:tilb} and \ref{it:tilc} hold with $e$ being $(i,1_M,i)$. 
Further, if $(i,a,j)\,\ehte\, (i,b,j)$, then with
$p=(j, 1_M,i), q=(i,1_M,j)$ we see that \ref{it:tild} also holds. 
Thus if $S$ is right coherent, then by Proposition~\ref{prop:thetechresult}, so is $M_i$.
But, since $M_i$ is isomorphic to $M$, we conclude that  $M$ is right coherent.
\end{proof}

The next proof, although the calculations are different, follows the exact pattern of the previous, and so we omit it.

\begin{Cor}\label{cor:br} 
Let $S=BR(M,\theta)$ be a Bruck--Reilly {extension of} a monoid $M$. If $S$ is right coherent, then so is $M$.
\end{Cor}

\begin{Rem}\label{rem:double} An easy adjustment to Corollary~\ref{cor:br} allows us to apply Proposition \ref{prop:thetechresult} to $EBR(M,\theta)^1$.
\end{Rem}

The monoids in Corollaries~\ref{cor:btdown}, \ref{cor:br} and Remark~\ref{rem:double}
all possess an {\em inverse skeleton} \cite{rida}; the techniques of \cite{rida} would provide an alternative unified approach to these results. 

\begin{qn} Let $S$ be a right coherent monoid. Theorem~\ref{thm:jclasses} tells us that if $J$ is a monoid $\jay$-class, then $J$ inherits right coherency from $S$. Is the same true for (i) $J^1$, any semigroup $\jay$-class 
with an identity adjoined? (ii)  any monoid principal factor $I$? or (iii) $I^1$ where $I$ is a semigroup principal factor with an identity adjoined? 
\end{qn}

\section{Brandt semigroups}\label{sec:up}

We saw in Proposition \ref{prop:cs1} that Rees matrix monoids over groups are right coherent.
We do not know whether the analogue of this result for Rees matrix monoids with zero holds.
We are, however, able to prove that it does for Brandt monoids, not only over groups but in fact over arbitrary right coherent monoids. This complements  Corollary \ref{cor:btdown}, and is the main result of this section:

\begin{Thm}\label{thm:brandt}
Let $M$ be a right coherent monoid and $I$ a non-empty set.
Then the Brandt monoid $\mathcal{B}(M;I)^1$ over $M$ is also right coherent.
\end{Thm}

The rest of this section will be devoted to the proof of Theorem \ref{thm:brandt}.
Throughout 
$M$ will denote an arbitrary right coherent monoid, 
$S=\mathcal{B}(M;I)^1$ a Brandt monoid over $M$, and $\rho$ a right congruence on $S$.
The identity of $M$ will be denoted by $1_M$, and that of $S$ simply by $1$.
The proof proceeds along the following lines:
\begin{txitemize}
\item
In Lemma \ref{lem:gens} we establish a classification of congruences on $S$, along with certain specific forms of their generating sets.
\item
In Lemma \ref{lem:passtofree} we establish a connection between a congruence on $S$ and a certain congruence on a free $M$-act.
\item
We verify the usual properties that ensure right coherency -- namely finite generation of annihilators and intersections -- in Lemmas \ref{lem:rarho} and \ref{lem:int} respectively.
\end{txitemize}

We single out one element of the index set $I$, denote it by $\bullet\,$, and note that the following holds:
\begin{equation}
\label{eq:bull}
(i,b,j)\,\rho\, (k,c,j)\Leftrightarrow (i,b,\bullet)\,\rho\, (k,c,\bullet)\quad \text{for all } i,j,k\in I,\ b,c\in M.
\end{equation}

\begin{Lem}\label{lem:gens} 
The right congruence $\rho$ satisfies precisely one of the following:  
\begin{thmenumrom}
\item\label{it:ct1}
$\rho=S\times S=\langle (1,0)\rangle_S$; or
\item\label{it:ct2}
$1\rho=\{ 1\}$ and $\rho$ has  a set of generators 
\[H=\bigl\{ \bigl((i,b,\bullet),(k,c,\bullet)\bigr)\::\:(i,b,k,c)\in A\bigr\} \cup\bigl\{  \bigl((j,d,\bullet),0\bigr)\::\: (j,d)\in B\bigr\},\]
for some suitable  index sets $A,B$;
or
\item\label{it:ct3}
$\{ 1\}\neq1\rho\neq S$ and $\rho$ has  a set of generators 
\begin{align*}
H= \bigl\{ \bigl((i,1_M,i),1\bigr)\bigr\} &\cup
\bigl\{ \bigl( (i,b,\bullet),(i,c,\bullet)\bigr)\::\: (i,b,i,c)\in A\bigr\}\\
&\cup \bigl\{ \bigl((j,d,\bullet),0\bigr)\::\: (j,d)\in B\bigr\},
\end{align*}
for some  $i\in I$ and index sets $A\neq \emptyset$ and $B$.
\end{thmenumrom}

Further, in \ref{it:ct2} and \ref{it:ct3},  if $\rho$ is finitely generated, we can take $A$ and $B$ to be finite.
\end{Lem}

\begin{proof} 
{Note that} $\rho=S\times S$ if and only if $(1,0)\in \rho$, and then $\rho$ is generated by
$\{ (1,0)\}$. We suppose therefore that $(1,0)\notin \rho$.

Choose a set of generators $K$ for $\rho$.  If
$\bigl((i,b,j),(k,c,\ell)\bigr)\in K$ and $j\neq \ell$, then 
\[(i,b,j)=(i,b,j)(j,1_M,j)\,\rho\, (k,c,\ell)(j,1_M,j)=0,\]
so that $\bigl((i,b,j),0\bigr),\bigl((k,c,\ell),0\bigr)\in \rho$. On the other hand, if
we are given that $\bigl((i,b,j),0\bigr),$ $\bigl((k,c,\ell),0\bigr)\in K$, then 
$\bigl((i,b,j),(k,c,\ell)\bigr)\in \rho$. Thus we may replace any pair in $K$ with differing third {coordinates} by two pairs
in which the second coordinate is $0$. 
{Thus}  we may take those pairs to be of the form $\bigl((j,d,\bullet),0\bigr)$.

By \eqref{eq:bull} we may assume that all the remaining pairs in $K$, not involving $0$ or $1$, are of the form
$\bigl((i,b,\bullet),(k,c,\bullet)\bigr)$. 
If there are no pairs involving $1$, we may obtain a set of generators
as in \ref{it:ct2}.

Finally we suppose that there is a pair in $K$ of the form $\bigl(1,(i,b,j)\bigr)$ or
$\bigl((i,b,j),1\bigr)$. 
Since $1\rho$ is a submonoid of $S$ {\em not} containing $0$, we see that
$i=j$ and if $\bigl(1,(k,c,k)\bigr)$ or $\bigl((k,c,k),1\bigr)\in K$, then $i=k$. We now observe that
\[ 1\,\rho\, (i,b,i)=(i,b,i)(i,1_M,i)\,\rho\, 1(i,1_M,i)=(i,1_M,i).\]
 We can therefore add $\bigl((i,1_M,i),1\bigr)$ to $K$ and replace any pair $\bigl(1,(i,b,j)\bigr)$ or
$\bigl((i,b,j),1\bigr)$ by $\bigl((i,1_M,\bullet),(i,b,\bullet)\bigr)$. 
By the preceding discussion and \eqref{eq:bull}, the congruence generated by the set will be unchanged, i.e. still equal to $\rho$.
Further, if we have a pair
$\bigl((j,b,\bullet),(k,c,\bullet)\bigr)\in K$, then since
\[(j,b,\bullet)=1(j,b,\bullet)\,\rho\, (i,1_M,i)(j,b,\bullet) \]
we deduce that either $(j,b,\bullet)\,\rho\, 0\,\rho\, (k,c,\bullet)$, so that we may replace
$\bigl((j,b,\bullet),(k,c,\bullet)\bigr)$ in $K$ by $\bigl((j,b,\bullet),0\bigr),\bigl((k,c,\bullet),0\bigr)$, or else $j=i$. We can perform the same manoeuvre for $k$.  Finally, if necessary to make the set $A$ non-empty, we can add a pair from the diagonal, and we  obtain a set of generators as in \ref{it:ct3}.

{It follows} from the above argument that if the original generating set $K$ was finite, then $H$ would be finite as well,
and the lemma is proved.
\end{proof}

From now on we assume that we fix a set of generators $H$ for $\rho$ in the form guaranteed by
 Lemma~\ref{lem:gens}. 
We define the {\em support of $A$}, denoted by $\supp A$,  in cases \ref{it:ct2} and \ref{it:ct3} above, 
to be the  set of elements of $I$ appearing as first or third {coordinates} in $A$. 
Let $\rho_A\subseteq \rho$ be the right congruence on $S$ generated by
\[H_A=\bigl\{ \bigl((i,b,\bullet),(k,c,\bullet)\bigr):(i,b,k,c)\in A\bigr\}.\]
It is easy to see that $\{ 1\}$ and $\{ 0\}$ are $\rho_A$-classes, and if two triples in $S$ are
$\rho_A$-related, they have the same third {coordinate}.

\begin{Lem}\label{lem:passtofree} 
Suppose that $\rho\neq S\times S$. 
Let $F_A$ be the free right $M$-act on 
$\{ x_h:h\in \supp A\}$, and let $\tau_A$ be the congruence on $F_A$ generated by
\[G_A=\{ (x_ib,x_kc):(i,b,k,c)\in A\}.\]
Then for any $(p,d),(q,e)\in I\times M$ we have
\[(p,d,\bullet)\,\rho_A\, (q,e,\bullet)\Leftrightarrow x_pd\,\tau_A\, x_{q}e.\]
\end{Lem}

\begin{proof} 
Suppose first that 
$x_pd\,\tau_A\, x_{q}e$. If $x_pd= x_{q}e$, {then} $(p,d,\bullet)=(q,e,\bullet)$. 
We suppose therefore that there is a $G_A$-sequence
\[x_pd= u_1t_1,\, v_1t_1=u_2t_2,\hdots, v_nt_n=x_{q}e\]
where $n\in\mathbb{N}$, $(u_j,v_j)\in \overline{G}_A$ and $t_j\in M$ for $1\leq j\leq n$; write
 $u_j=x_{i_j}b_j, v_j=x_{k_j}c_j$. 

Then
\begin{multline*}
(p,d,\bullet)=(i_1, b_1, \bullet)(\bullet,t_1,\bullet)\,\rho_A\,
(k_1, c_1,\bullet)(\bullet,t_1,\bullet)
=(i_2, b_2,\bullet)(\bullet,t_2,\bullet)\\
\rho_A\,
(k_2, c_2,\bullet)(\bullet,t_2,\bullet)\,\rho_A\, \hdots \, \rho_A\, 
(k_n,c_n,\bullet)(\bullet,t_n,\bullet)=(q,e,\bullet).
\end{multline*}

For the reverse implication, suppose $(p,d,\bullet)\,\rho_A\, (q,e,\bullet)$.
If $(p,d,\bullet) = (q,e,\bullet)$ then 
$x_pd\,\tau_A\, x_{q}e$. Otherwise, we have a $\rho_A$-sequence
\[(p,d,\bullet)=\tuple{b}_1 \tuple{m}_1, \tuple{c}_1\tuple{m}_1=\tuple{b}_2\tuple{m}_2,\hdots, \tuple{c}_n\tuple{m}_n=(q,e,\bullet)\]
where $n\in\mathbb{N}$, $(\tuple{b}_j,\tuple{c}_j)\in \overline{H}_A$ and $\tuple{m}_j\in S$ for $1\leq j\leq n$; write
$\tuple{b}_j=(i_j,b_j,\bullet)$, $\tuple{c}_j=(k_j,c_j,\bullet)$. Since each $\tuple{b}_j,\tuple{c}_j$ have common third coordinate, {the element} $0$ does not appear in the above sequence.  By replacing any $\tuple{m}_j=1$ by $(\bullet,1_M,\bullet)$, we can assume that
$\tuple{m}_j=(\bullet ,m_j,\bullet)$ for $1\leq j\leq n$.  Then
\[x_pd=x_{i_1}b_1m_1\,\tau_A\, x_{k_1}c_1m_1=x_{i_2}b_2m_2\,\tau_A\, \hdots\,\tau_A\, x_{k_n}c_nm_n=x_{q}e,\]
thus completing the proof of the lemma.
\end{proof}

We now turn to right annihilators.

\begin{Lem}\label{lem:rarho} 
If $\rho$ is finitely generated then for any $\tuple{a}\in S$ the right annihilator
$\ann(\tuple{a}\rho)$ is finitely generated.
\end{Lem}

\begin{proof} 
{Note that} $\ann(1\rho)=\rho$ and $\ann(0\rho)=S\times S$, which are both finitely generated. 
So we assume that $\tuple{a}=(u,a,v)$ and
that
$\ann(\tuple{a}\rho)\neq S\times S$; in particular $(\tuple{a},0)\notin \rho$. If there is a $\bigl((i,1_M,i),1\bigr)\in H$ and $u\neq i$ then $\tuple{a}\,\rho\, 0$, a contradiction. Thus, we assume that if we are in case \ref{it:ct3}, then $u=i$.

To construct a finite set of generators of $\ann(\tuple{a}\rho)$, first choose a finite set $T$ of generators of $\ann(a)$ and let
\[R_1=\bigl\{ \bigl((v,x,v),(v,y,v)\bigr)\::\: (x,y)\in T\bigr\} \cup \bigl\{ \bigl(1,(v,1_M,v)\bigr)\bigr\};\]
clearly, $R_1\subseteq \ann(\tuple{a}\rho)$. 

If $u$ is in the support of $A$, then,  by Theorem \ref{thm:chasetype}, there is a finite set of generators $U$ of
$\ann((x_ua)\tau_A)$. For any $(s,t)\in U$ we have $x_uas\,\tau_A\, x_uat$, and it follows from
Lemma~\ref{lem:passtofree} that $R_2\subseteq \ann(\tuple{a}\rho)$, where
\[R_2=\bigl\{ \bigl((v,s,v),(v,t,v)\bigr)\::\:(s,t)\in U\bigr\}.\]

Still supposing that $u$ is in the support of $A$, let us consider a pair $\alpha=(j,d)\in B$. 
If $(x_ua)\tau_AM\cap (x_jd)\tau_AM\neq\emptyset$ (which can only happen if $j$ is also in the support of $A$), then, again by Theorem \ref{thm:chasetype}, we have
that 
\[(x_ua)\tau_AM\cap (x_jd)\tau_AM=\bigcup_{p\in N(\alpha)}(x_{\alpha,p}w_{\alpha,p})\tau_AM\]
for some finite set $N(\alpha)$. For any $p\in N(\alpha)$, there are elements
$h_{\alpha,p}, k_{\alpha,p}\in M$ such that
\[x_uah_{\alpha,p}\,\tau_A\, x_{\alpha,p}w_{\alpha,p}\,\tau_A\, x_jdk_{\alpha,p}.\]
By Lemma~\ref{lem:passtofree} we have that
\[(u,ah_{\alpha,p},\bullet)\,\rho_A\,  \bigl(j,dk_{\alpha,p},\bullet\bigr)\]
whence
\[(u,a,v)(v,h_{\alpha,p},\bullet)=(u,ah_{\alpha,p},\bullet)\,\rho\,  (j,d,\bullet)(\bullet,k_{\alpha,p},\bullet)\,\rho\, 0\]
and so
\[\bigl((v,h_{\alpha,p},v),0\bigr)\in \ann(\tuple{a}\rho).\]
Thus, letting
\[R_3=
\bigl\{\bigl((v,h_{\alpha,p},v),0\bigr)\::\: \alpha\in B, p\in N(\alpha)\bigr\},\]
we have that $R_3\subseteq \ann(\tuple{a}\rho)$.

We now show that $R=R_1\cup R_2\cup R_3$ generates $\ann(\tuple{a}\rho)$. 
We let $\nu=\langle R\rangle$, the right congruence of $S$ generated by $R$.

We have already observed that $R_1\cup R_2\cup R_3\subseteq \ann(\tuple{a}\rho)$, and so 
$ \nu\subseteq \ann(\tuple{a}\rho)$. 
For the reverse inclusion
suppose $(\tuple{r},\tuple{s})\in \ann(\tuple{a}\rho)$;  we aim to show that $(\tuple{r},\tuple{s})\in\nu$. 

First consider the case where $\tuple{a}\tuple{r}=\tuple{a}\tuple{s}=0$, recalling $\tuple{a}=(u,a,v)$.
Then we must have either $\tuple{r}=0$, or else $\tuple{r}=(k,r,\ell)$ with $v\neq k$.
In the latter case we have
\[
(k,r,\ell)=1(k,r,\ell)\,\nu\, (v,1_M,v)(k,r,\ell)=0.
\]
Hence  $\tuple{r}\, \nu\, 0$;  by symmetry $\tuple{s}\, \nu\, 0$, and therefore, by transitivity, 
$(\tuple{r},\tuple{s})\in\nu$.

Next consider $\tuple{a}\tuple{r}=\tuple{a}\tuple{s}\neq 0$.
If both $\tuple{r}$ and $\tuple{s}$ are triples, say 
$\tuple{r}=(k,r,\ell)$, $\tuple{s}=(h,s,m)$, then $v=k=h, \ell=m$ and 
using $R_1$ we have that $(\tuple{r},\tuple{s})\in\nu$. 
If only one of them is a triple, say $\tuple{r}=(k,r,\ell)$ and $\tuple{s}=1$, then from
$(u,a,v)(k,r,\ell)=(u,a,v)1$, we have  $k=\ell=v$, and as $(u,a,v)(v,1_M,v)=(u,a,v)(v,r,v)$, we have that 
$1\,\nu\, (v,1_M,v)\, \nu\, (v,r,v)=\tuple{r}$, i.e. again $(\tuple{r},\tuple{s})\in\nu$.
Finally if $\tuple{r}=\tuple{s}=1$ then it is obvious that
$(\tuple{r},\tuple{s})\in\nu$.

Finally, consider the case where $\tuple{a}\tuple{r}\neq \tuple{a}\tuple{s}$.
Since $(\tuple{a}\tuple{r}, \tuple{a}\tuple{s})\in \rho$ there  must be an $H$-sequence
\begin{equation}
\label{eq:annHseq}
\tuple{a}\tuple{r}=(u,d,\ell)=\tuple{c}_1\tuple{t}_1, \tuple{d}_1\tuple{t}_1=\tuple{c}_2\tuple{t}_2,\hdots ,\tuple{d}_n\tuple{t}_n=(u,e,m)=\tuple{a}\tuple{s},
\end{equation}
where $(\tuple{c}_j,\tuple{d}_j)\in \overline{H}$, $\tuple{t}_j\in S$ for $1\leq j\leq n\in\mathbb{N}$. 

Let us first assume that $0$ does not occur in the sequence \eqref{eq:annHseq}. 
Immediately, we deduce that $u\in \supp A$. 
If we are in case \ref{it:ct3}, where a generator $\bigl( (i,1_M,i),1\bigr)$ lies in $H$, then we note that $\supp A=\{ i\}$ and $u=i$. Then by multiplying every element of the above sequence {\em on the left} by $(i,1_M,i)$, we may obtain a (possibly shorter) sequence connecting $\tuple{a}\tuple{r}$ and
$\tuple{a}\tuple{s}$, in which $0$ does not appear and the generating pairs from $\overline{H}$ occurring are all of the form  
$\bigl((i,b,\bullet),(i,c,\bullet)\bigr)\in H$, where $(i,b,i,c)\in A$. Thus, we can assume, in both cases \ref{it:ct2} and \ref{it:ct3}, that $(\tuple{c}_j,\tuple{d}_j)\in \overline{H}_A$. It is then easy to see that the third {coordinate} of every element of the sequence \eqref{eq:annHseq} is equal to $\ell$ (in particular $m=\ell$), so that multiplying {\em on the right} by $(\ell,1_M,\ell)$, we can assume that all the $\tuple{t}_j$s are triples.

As $0$ is not in the sequence \eqref{eq:annHseq}, and as
$(u,a,v)1=(u,a,v)(v,1_M,v)$ and $\bigl(1,(v,1_M,v)\bigr)\in R$, we can assume without loss of generality that 
both $\tuple{r}$ and $\tuple{s}$ are triples, say $\tuple{r}=(v,r,\ell)$, $\tuple{s}=(v,s,\ell)$. 
From
     $(u,ar, \ell)\,\rho_A\, (u,as,\ell)$ we have $(u,ar, \bullet)\,\rho_A\, (u,as,\bullet)$ by \eqref{eq:bull},
     and then
by Lemma~\ref{lem:passtofree} we have $x_uar\,\tau_A\, x_uas$. Thus
$(r,s)  \in \ann((x_ua)\tau_A) $, which is generated by $U$.
 Therefore it is easy  to see (using the generators in $R_2$) that
$(\tuple{r},\tuple{s})\in\nu$.

Finally, we consider the case where $0$ appears somewhere in the sequence \eqref{eq:annHseq}.
We claim that
\begin{equation}
\label{eq:rnu0}
(\tuple{r},0)\in\nu.
\end{equation}
The assertion is obvious if $\tuple{r}=0$. 
Also $\tuple{r}\neq 1$ because $\ann(\tuple{a}\rho)\neq S\times S$.
We can therefore assume that  $\tuple{r}=(w,r,\ell)$. If $w\neq v$, then 
\[
(w,r,\ell)=1(w,r,\ell)\ \nu\ (v,1_M,v)(w,r,\ell)=0,
\]
proving \eqref{eq:rnu0} in this case. 

Suppose now that $w=v$, and consider the first occurrence of $0$ in \eqref{eq:annHseq}, say
  $\tuple{c}_1\tuple{t}_1, \hdots, \tuple{c}_i\tuple{t}_i\neq 0$ and $\tuple{d}_i\tuple{t}_i=0$. 
  Noting that $\tuple{c}_i\tuple{t}_i\neq 1$, we can assume  as above that 
\[
(u,a,v)(v,r,\ell)\,\rho_A\, \tuple{c}_i\tuple{t}_i.
\]
Since $\tuple{d}_i\tuple{t}_i=0$, we are forced to have
\[
\tuple{c}_i=(j,d,\bullet),\ \tuple{d}_i=0\quad 
\text{for some } \alpha=(j,d)\in B.
\]
 We can therefore take $\tuple{t}_i=(\bullet, x,\ell)$ and notice that
\[(u,ar,\bullet)\,\rho_A\, (j,dx,\bullet),\]
whence $x_uar\,\tau_A\, x_jdx$. Thus
\[(x_ua)\tau_AM\cap (x_jd)\tau_AM\neq \emptyset\]
so that
\[x_uar\,\tau_A\, x_jdx\,\tau_A\, x_{\alpha,p}w_{\alpha,p}y\]
for some $p\in N(\alpha)$ and $y\in M$. Notice that
\[x_uar\,\tau_A\, x_{\alpha,p}w_{\alpha,p}y\,\tau_A\, x_uah_{\alpha,p}y,\]
so that using $R_2$ we have
\[\bigl((v,r,\ell),(v,h_{\alpha,p}y,\ell)\bigr)\in \nu.\]
Then using $R_3$,
\[\mathbf{r}=(v,r,\ell)\,\nu\, (v,h_{\alpha,p}y,v)=
(v,h_{\alpha,p},v)(v,y,\ell)\,\nu\, 0,\]
completing the proof of \eqref{eq:rnu0}. 

Analogously to \eqref{eq:rnu0} we have $(\tuple{s},0)\in\nu$, and hence $(\tuple{r},\tuple{s})\in \nu$
by transitivity. This completes the proof that $\ann(\tuple{a}\rho)=\nu$, and hence of the lemma.
\end{proof}

The following lemma deals with intersections:

\begin{Lem}
\label{lem:int} 
If $\rho$ is finitely generated then for any  $\tuple{a},\tuple{b}\in S$ we have that 
$(\tuple{a}\rho) S \cap (\tuple{b}\rho) S$ is finitely generated.
\end{Lem}

\begin{proof} 
If $\tuple{a}$ or $\tuple{b}$ is $\rho$-related to $1$ or $0$, or if $\tuple{a}=\tuple{b}$, then 
$(\tuple{a}\rho)S\cap(\tuple{b}\rho)S$ is monogenic. 
Suppose therefore that $\tuple{a}\neq \tuple{b}$ and that both $\tuple{a}$ and $\tuple{b}$ are triples,
say $\tuple{a}= (u,a,v)$, $\tuple{b}=(w,b,z)$.

If $(u,a,v)S\cap (w,b,z)S\neq \{ 0\}$, then $u=w$ and $aM\cap bM\neq\emptyset$. 
Since $M$ is right coherent, $aM\cap bM=\bigcup_{c\in C}cM$ for some finite set
$C$, and then
\[
U=\bigcup_{c\in C}(u,c,u)\rho\subseteq (\tuple{a}\rho)S\cap(\tuple{b}\rho)S.
\]

Further, provided that $u,w\in \supp A$,  
with the $M$-act $F_A/\tau_A$ we consider the
 intersection
$((x_ua)\tau_A)M\cap ((x_wb)\tau_A)M\neq\emptyset$.
Since $M$ is right coherent, this intersection is finitely generated by Theorem \ref{thm:chasetype},
i.e.
there exists a finite set $D\subseteq M$ such that
\begin{equation}
\label{eq:intgen}
((x_ua)\tau_A)M\cap ((x_wb)\tau_A)M=\bigcup_{d\in D}((x_uad)\tau_A)M.
\end{equation} 
If we let
\[
V=\{ (u,ad,\bullet)\rho\::\: d\in D\}
\]
then it follows from \eqref{eq:intgen} and Lemma~\ref{lem:passtofree} that
\[V\subseteq (\tuple{a}\rho)S\cap(\tuple{b}\rho)S.\]

Let $W$ be the subact of $S/\rho$ generated by  $U\cup V$; we claim that  $W=(\tuple{a}\rho)S\cap(\tuple{b}\rho)S$. 
By the preceding discussion, we  have $W\subseteq (\tuple{a}\rho)S\cap(\tuple{b}\rho)S$.
For the reverse inclusion, first note that
 $0\rho\in (\tuple{a}\rho)S\cap(\tuple{b}\rho)S$ always, and if 
$ (\tuple{a}\rho)S\cap(\tuple{b}\rho)S=\{0\rho\}$ there is nothing to prove.
So suppose now that there exist $\tuple{h},\tuple{g}\in S$ such that $(\tuple{a}\tuple{h})\rho=(\tuple{b}\tuple{g})\rho\neq 0\rho$;
we want to show that $(\tuple{a}\tuple{h})\rho\in W$.

Now, from $(\tuple{a}\tuple{h},\tuple{b}\tuple{g})\in\rho$ we have that either
$\tuple{a}\tuple{h}=\tuple{b}\tuple{g}\neq 0$ or else there exists an $H$-sequence from one to the other.
In the former case $u=w$, $aM\cap bM\neq \emptyset$, and
$(\tuple{a}\tuple{h})\rho\in US$.

Now suppose that there exists an $H$-sequence
\begin{equation}
\label{eq:ahbg}
(u,a,v)\tuple{h}=\tuple{c}_1\tuple{t}_1, \tuple{d}_1\tuple{t}_1=\tuple{c}_2\tuple{t}_2,\hdots ,\tuple{d}_n\tuple{t}_n=(w,b,z)\tuple{g},
\end{equation}
where $(\tuple{c}_i,\tuple{d}_i)\in \overline{H}$, $\tuple{t}_i\in S$ for $1\leq i\leq n\in\mathbb{N}$.
Since $(\tuple{a}\tuple{h})\rho\neq 0\rho$, no element in the sequence equals $0$. 
We can assume as in Lemma~\ref{lem:rarho} that every pair 
$(\tuple{c}_i,\tuple{d}_i)\in \overline{H}_A$ and every $\tuple{t}_i$ is a triple. 
{Note that} $u,w\in \supp A$. 
We know that $\tuple{h},\tuple{g}\neq 0$.
If either of them were $1$, 
then 
$(\tuple{a}\rho)S\cap(\tuple{b}\rho)S$ would be principal, and our job would be done. 
We can therefore assume that $\tuple{h}=(v,h,y)$, $\tuple{g}=(z,g,y)$. 
Then the sequence \eqref{eq:ahbg} implies
\[
(u,a,v)(v,h,y)\,\rho_A\, (w,b,z)(z,g,y)
\]
giving that
\[
(u,ah,\bullet)\,\rho_A\, (w,bg,\bullet)
\]
using \eqref{eq:bull},
and consequently $x_uah\,\tau_A\, x_{w}bg$ by Lemma \ref{lem:passtofree}. 
Therefore 
\[
(x_uah)\tau_A=(x_{w}bg)\tau_A\in (x_ua)\tau_AM\cap (x_{w}b)\tau_AM.
\]
By the definition of the set $D$ it follows that there exists $d\in D$ and $m\in M$ such that
$(x_uah)\tau_A=(x_uadm)\tau_A$.
By Lemma~\ref{lem:passtofree} and \eqref{eq:bull} we now have
$(u,ah,y)\,\rho \, (u,adm,y)$, and hence
\[
(\tuple{a}\tuple{h})\rho =(u,ah,y)\rho=\bigl((u,ad,\bullet)\rho\bigr)(\bullet,m,y)\in VS.
\]
This completes the proof of $(\tuple{a}\tuple{h})\rho\in W$,
and of the entire lemma.
\end{proof}

\begin{proof}[Proof of Theorem \ref{thm:brandt}]
The theorem follows immediately from Lemmas \ref{lem:rarho} and \ref{lem:int}, using Theorem
\ref{thm:chasetype}.
\end{proof}

As a consequence we can prove that right coherency is preserved by adjoining or removing a zero element from a monoid, paralleling Proposition \ref{cor:1} for identity elements.

\begin{Cor} 
\label{cor:0}
A monoid $M$ is right coherent if and only if $M^{\underline{0}}$ is right coherent.
\end{Cor}

\begin{proof} 
($\Rightarrow$)
Note that $(M^{\underline{0}})^{\underline{1}}$ is isomorphic to the $1\times 1$ Brandt monoid  $\mathcal{B}(M;\{1\})$.
Hence if $M$ is right coherent then so is $(M^{\underline{0}})^{\underline{1}}$
 by Theorem \ref{thm:brandt}. But then it follows that $M^{\underline{0}}$ is right coherent by Proposition \ref{cor:1}.

($\Leftarrow$)
Suppose that $M^{\underline{0}}$ is right coherent, and let $\rho$ be a finitely generated right congruence on $M$. Then 
$\rho^0=\rho\cup \{ (0,0)\}$
is a finitely generated right congruence on $M^{\underline{0}}$. Let $a\in M$. It is easy to see that
$\ann(a\rho^0)=\ann(a\rho)\cup \{ (0,0)\}$. Let $H\subseteq M\times M$ be a finite set such that
$H\cup  \{ (0,0)\}$ generates $\ann(a\rho^0)$. Since $M^{\underline{0}}$ has no zero divisors, a standard argument gives that   $H$ generates
$\ann(a\rho)$.

Now let $a,b\in M$ and consider $(a\rho^0)M^{\underline{0}}\cap (b\rho^0)M^{\underline{0}}$. For any $u,v\in M^{\underline{0}}$ we have $au\,\rho^0\, bv$ if and only if $au=bv=0$ or $au,bv\in M$ and $au\,\rho\, bv$.
By assumption,
$(a\rho^0)M^{\underline{0}}\cap (b\rho^0)M^{\underline{0}}$ is finitely generated, say by $\{x\rho^0:x\in X\} \cup \{ 0\rho^0\}$, where $X$ is a finite subset of $M$. It is then easy to show, bearing in mind that $0\rho^0=\{ 0\}$,  that $\{ x\rho:x\in X\}$ generates $(a\rho)M\cap (b\rho)M$.
\end{proof}

Unfortunately, we do not know whether the analogues of the above result hold for arbitrary Rees matrix semigroups with or without zero, or for Bruck--Reilly extensions.

\begin{qn} 
If $M$ is right coherent monoid, are the following necessarily right coherent:
(a) an arbitrary Rees matrix monoid $\mathcal{M}(M; I,\Lambda;P)^1$; 
(b) an arbitrary Rees matrix monoid with zero $\mathcal{M}^0(M; I,\Lambda;P)^1$;
(c) an arbitrary Bruck--Reilly {extension} $\BR(M,\theta)$?
\end{qn}

\section{Direct products}\label{sec:dp}

In this section we initiate the investigation into when the direct product of two monoids is right coherent.
We begin with the following easy observation:

\begin{Prop}
If the direct product $S\times T$ of two monoids is right coherent then so are both $S$ and $T$.
\end{Prop}

\begin{proof}
The mapping $S\times T\rightarrow S\times\{1_T\}\cong S$, $(s,t)\mapsto (s,1_T)$ is a retraction, and hence, by Proposition~\ref{prop:retracts}, $S$ is right coherent. The proof for $T$ is analogous. 
\end{proof}

Next we observe that the converse does not hold: it is possible for the direct product of two right coherent monoids not to be right coherent.

\begin{Ex}\label{ex:direct}
Let $F$ be the free monoid generated by $\{a,x,b\}$.
By \cite{FMonoid}, $F$ is right coherent, and we claim that $F\times F$ is not right coherent.

To this end, let $\rho$ be the right congruence on $F\times F$ generated by
\[
H=\bigl\{ \bigl((ax,1),(a,x)\bigr), \bigl((ab,1),(aa,1)\bigr)\bigr\}.
\]
We show that the annihilator $\nu=\ann\bigl( (a,b)\rho\bigr)$ is not finitely generated.
To see this note that for every $n\geq 1$ we have
\begin{equation}
\label{eq:axnb}
(ax^n,b) \mathrel{\rho} (ax^{n-1},xb) \mathrel{\rho} \ldots \mathrel{\rho} (a,x^nb).
\end{equation}
What is more, the elements appearing in \eqref{eq:axnb} form a complete $\rho$-class, 
because no further rules from $H$ can be applied to any of them.
It follows that $(x^n,1)\nu=\{(x^n,1)\}$ for all $n\geq 1$, and it is easy to see this is also true for $n=0$.

In a similar fashion, we have
\[
(ax^nb,b) \mathrel{\rho} (ax^{n-1}b,xb) \mathrel{\rho} \ldots \mathrel{\rho} (ab,x^nb) \mathrel{\rho} (aa,x^nb) \mathrel{\rho} (axa,x^{n-1}b) \mathrel{\rho} \ldots \mathrel{\rho} (ax^na,b),
\]
and it follows that $(x^nb,1)\, \nu\, (x^na,1)$.

Now, let $K$ be an arbitrary finite set of pairs from $F \times F$, and let $n$ be a natural number greater than the length of any words appearing in any component of any element of $K$.
It follows that no element of $K$ has $x^na$ as a component. On the other hand, if $(x^na,1)$ is written as
$(u,v)(w,z)$, with $(u,v)$ a component of a pair in $K$,  then $(u,v)=(x^i,1)$ for some $i\geq 0$. It follows from the discussion above that $(u,v)$ is not $\nu$-related to any element of $F\times F$ other than itself and so there does not exist a  
$K$-sequence starting from $(x^na,1)$ and ending in $(x^nb,1)$. Therefore $K$ cannot generate $\nu$,  and we deduce that
$\nu$ is not finitely generated. Hence  $F\times F$ is not right coherent.
\end{Ex}

In Example~\ref{ex:direct}, certainly $F\times F$ is a submonoid of $G\times G$, where $G$ is the free group on $\{ a,b,x\}$. {By \cite{gould:1992}, groups} are coherent.

\begin{Cor}\label{cor:submonoids} The class of right coherent monoids is not closed under submonoids.
\end{Cor}

By way of contrast to Example~\ref{ex:direct}, direct products with finite monoids do preserve right coherency:

\begin{Thm}
\label{thm:SxTfin}
If $S$ is a right coherent monoid and $T$ is a finite monoid then $S\times T$ is right coherent.
\end{Thm}

We will prove this by reference to the original definition of coherency. The main work in doing so is contained in the following result, which may be of independent interest. In what follows, we denote the identities of $S$ and $T$ by $1_S$ and $1_T$, respectively.

\begin{Prop} \label{prop:direct}
Let $S$ be a monoid,  let  $T$ be a finite monoid and let $A$ be an $S \times T$-act.
Then $A$ is a right $S$-act with the action  $as=a(s,1_T)$ for all $a \in A$ and $s \in S$.
Further:

\begin{thmenumrom}  
\item \label{it:xi}
$A$ is finitely generated as an $S$-act if and only if it is finitely generated as an $S\times T$-act;
\item \label{it:xii}
$A$ is finitely presented as an $S$-act if and only if it is finitely presented as an $S\times T$-act.
\end{thmenumrom}
\end{Prop}

\begin{proof}
The first statement is straightforward to verify.
\medskip

\ref{it:xi}
($\Rightarrow$)
If $A$ is finitely generated as an $S$-act by a set $X$, then $X$ also generates $A$ as an $S\times T$-act.
\medskip

($\Leftarrow$)
If $A$ is finitely generated as an $S\times T$-act by $X$, then for every $a \in A$ there exist $(s,t) \in S \times T$ and $x \in X$ such that $a=x(s,t)=x(1_S,t)(s,1_T)$, which shows that $A$ is finitely generated as an $S$-act by the set $\{x(1_S,t)\::\: x\in X,\ t \in T\}$.
\medskip

\ref{it:xii} ($\Rightarrow$)
Suppose  that $A$ is finitely presented as an $S$-act. Let $U$ be a finite generating set for   $A$ as an $S$-act such that
$u(1_S,t)\in U$ for all $u\in U$ and $t\in T$. Let  ${X}=\{ x_u\: :\:u\in U\}$. Since finite presentability does not depend on the choice of the generating set, there is an $S$-morphism $\theta:F_S({X}){\rightarrow}A$, extending the map
$x_u\mapsto u$, such that $\ker\theta$ is generated by a finite 
set of  pairs $H\subseteq F_S({X})\times F_S({X})$.

Let $\psi:F_{S\times T}({X})\rightarrow A$ be the $S\times T$-morphism extending the map $x_u\mapsto u$
and let $H'\subseteq F_{S\times T}({X})\times F_{S\times T}({X})$ be defined by
\[
H'=\big\{ \big(x_u(s,1_T),x_v(t,1_T)\big) \::\: (x_us,x_vt)\in H\big\}
 \cup \bigl\{\big(x_{u}(1_S,t),x_{u(1_S,t)}\big): x_u \in X, t \in T \bigr\}.
\]

By construction, $H'\subseteq \ker \psi$; we claim that $H'$ generates $\ker\psi$. 
To this end, suppose that  $(x_u(s,t))\psi=(x_v(h,k))\psi$, so that $u(s,t)=v(h,k)$ in $A$. 
Then $u(1_S,t)(s,1_T)=v(1_S,k) (h,1_T)$ so that, regarding $A$ as an $S$-act, we  have
$\big(u(1_S,t)\big)s=\big(v(1_S,k)\big)h$. It follows that $(x_{u(1_S,t)}s)\theta=(x_{v(1_S,k)}h)\theta$ so that, as
$H$ generates $\ker\theta$, we have an $\overline{H}$-sequence
\[
x_{u(1_S,t)}s = (x_{p_1}c_1)t_1, (x_{q_1}d_1)t_1=(x_{p_2}c_1)t_1,
\ldots, (x_{q_n}d_n)t_n= x_{v(1_S,k)}h
\]
where $(x_{p_i}c_i,x_{q_i}d_i)\in \overline{H}$ and $t_i\in S$, $1\leq i\leq n$.
We therefore have an $\overline{H'}$-sequence
\[x_u(s,t)=x_u(1_S,t)(s,1_T),\, x_{u(1_S,t)}(s,1_T)=
x_{p_1}(c_1,1_T)(t_1,1_T), \,
x_{q_1}(d_1,1_T)(t_1,1_T)=\]\[x_{p_2}(c_2,1_T)(t_2,1_T), 
\hdots, x_{q_n}(d_n,1_T)(t_n,1_T)=x_{v(1_S,k)}(h,1_T),\, x_v(1_S,k)(h,1_T) =x_v(h,k)\] connecting ${x_u}(s,t)$ and $x_v(h,k)$. Thus $\ker\psi$ is generated by $H'$, giving  that  $A$ is 
finitely presented as an $S\times T$-act. 
\medskip

($\Leftarrow$)
Suppose that $A$ is finitely presented as an $S \times T$-act. Let $U$ be a finite generating set 
for $A$ as an $S\times T$-act, let $X=\{ x_u:u\in U\}$ and let
$\theta:F_{S\times T}(X)\rightarrow A$ be the $S\times T$-morphism extending the map
$x_u\mapsto u$. By assumption, $\ker\theta$ is finitely generated, say by
$H\subseteq F_{S\times T}(X)\times F_{S\times T}(X)$, where $H$ is finite.
Let 
\[X'=\{ x_{u(1_S,t)}:u\in U, t\in T\}\]
and put
\[
H'=\bigl\{ \bigl(x_{u(1_S,c't)}c,x_{v(1_S,d't)}d\bigr) \::\:
(x_u(c,c'),x_v(d,d'))\in H, t\in T\bigr\}.\]
Let $\psi:F_S(X')\rightarrow A$ be the $S$-morphism extending the map
$x_{u(1_S,t)}\mapsto u(1_S,t)$.  It is easy to check that $\psi$ is onto and  $H'\subseteq \ker \psi$.

To show that $\ker\psi$ is generated by $H'$, suppose that $\big(x_{u(1_S,t)}s\big)\psi=
\big(x_{v(1_S,k)}h\big)\psi$, so that $u(s,t)=v(h,k)$ in $A$ and hence 
$\big( x_u(s,t)\big)\theta=\big(x_v(h,k)\big)\theta$. Since $\ker\theta$ is
generated by $H$, we have an $\overline{H}$-sequence
\begin{multline*}
x_u(s,t)=x_{p_1}(c_1,c_1')(s_1,t_1), \, x_{q_1}(d_1,d_1')(s_1,t_1)
=x_{p_2}(c_2,c_2')(s_2,t_2),\\
\hdots, x_{q_n}(d_n,d_n')(s_n,t_n)=x_v(h,k),
\end{multline*}
where $\big(x_{p_i}(c_i,c_i'), x_{q_i}(d_i,d_i')\big)\in  \overline{H} $
and $(s_i,t_i)\in S\times T$, for $1\leq i\leq n$. We therefore have an $\overline{H'}$-sequence
\begin{multline*}
x_{u(1_S,t)}s=(x_{p_1(1_S,c_1't_1)}c_1)s_1,\, 
x_{q_1(1_S,d_1't_1)}d_1)s_1=(x_{p_2(1_S,c_2't_2)}c_2)s_2,\\
\hdots, (x_{q_n(1_S,d_n't_n)}d_n)s_n=x_{v(1_S,k)}h,
\end{multline*}
where $(x_{p_i(1_S,c_i't_i)}c_i,
x_{q_i(1_S,d_i't_i)}d_i)\in \overline{H'}$ and $s_i\in S$ for $1\leq i\leq n$, connecting $x_{u(1_S,t)}s$ and $x_{v(1_S,k)}h$. It follows that $\ker\psi$ is finitely generated and hence $A$ is finitely presented as an $S$-act.
\end{proof}

\begin{proof}[Proof of Theorem \ref{thm:SxTfin}]
Let $A$ be a finitely presented $S \times T$-act and $B$ a finitely generated $S\times T$-subact of $A$.
Then, by Theorem \ref{prop:direct}, $A$ is also a finitely presented $S$-act, and $B$ is finitely generated as an $S$-subact of $A$.
Since $S$ is right coherent we have that $B$ is finitely presented as an $S$-act, and then by Theorem \ref{prop:direct} we conclude that $B$ is also finitely presented as an $S \times T$-act, 
concluding that $S \times T$ is indeed right coherent.
\end{proof}

\begin{qn}Let $S$ be a right coherent monoid. For which classes of monoids $\mathcal{C}$ is it the case that
$S\times T$ is right coherent for all $T\in\mathcal{C}$?
\end{qn}

Theorem~\ref{prop:direct} tells us that the class of finite monoids is one example of a class $\mathcal{C}$ as above.

\section*{Acknowledgement} 

The authors would like to thank a careful referee for helpful and constructive comments, which have resulted in an improved version of our article.

\end{document}